\documentclass[11pt]{amsart}
\usepackage[utf8]{inputenc}

\usepackage[T1]{fontenc}
\usepackage[english]{babel}
\usepackage{amsthm, amssymb, amsfonts, amsmath, stmaryrd}
\usepackage{bbold}
\usepackage{xfrac}
\usepackage[all]{xy}
\usepackage{geometry}
\usepackage{hyperref}
\usepackage{graphicx}
\usepackage{tikz}
\usepackage{tikz-cd}
\usepackage{ytableau}
\usetikzlibrary{decorations.pathmorphing,arrows,arrows.meta,calc,shapes,shapes.geometric,decorations.pathreplacing}
\tikzcdset{arrow style=tikz, diagrams={>=stealth}}
\usepackage{color}
\usepackage{indentfirst}
\usepackage[stdsubgroups,nocfg]{nomencl}
\usepackage{Arbres}

\newcommand{\fonction}[3]{#1 : #2 \longrightarrow #3}
\newcommand{\Fonction}[5]{\begin{align*} #1 : #2 &\longrightarrow #3 \\ #4 &\longmapsto #5 \end{align*}}
\newcommand{\bb}[1]{\mathbb{#1}}
\newcommand{\mcal}[1]{\mathcal{#1}}

\newcommand{\nor}[1]{\left \lvert #1 \right \rvert}

\newcommand{\eq}[1][r]
   {\ar@<-3pt>@{-}[#1]
    \ar@<-1pt>@{}[#1]|<{}="gauche"
    \ar@<+0pt>@{}[#1]|-{}="milieu"
    \ar@<+1pt>@{}[#1]|>{}="droite"
    \ar@/^2pt/@{-}"gauche";"milieu"
    \ar@/_2pt/@{-}"milieu";"droite"}
\newcommand{\nocontentsline}[3]{}
\newcommand{\tocless}[2]{\bgroup\let\addcontentsline=\nocontentsline#1{#2}\egroup}

\theoremstyle{definition}
\newtheorem{defin}{Definition}[section]
\newtheorem{ex}[defin]{Example}
\newtheorem{exs}[defin]{Examples}

\newtheorem{rem}[defin]{Remark}

\newtheorem{nota}[defin]{Notation}
\newtheorem{notas}[defin]{Notations}
\theoremstyle{plain}
\newtheorem{theo}[defin]{Theorem}
\newtheorem{coro}[defin]{Corollary}
\newtheorem{prop}[defin]{Proposition}
\newtheorem{lemme}[defin]{Lemma}
\newtheorem{conj}[defin]{Conjecture}
\newtheorem{property}[defin]{Property}

\DeclareMathOperator{\spn}{Span}

\DeclareMathOperator{\id}{id}
\DeclareMathOperator{\rk}{rk}
\DeclareMathOperator{\op}{op}

\DeclareMathOperator{\Bimod}{Bimod}

\DeclareMathOperator{\norm}{norm}

\DeclareMathOperator{\Card}{Card}

\usepackage[backend = biber, style = alphabetic, maxbibnames = 4]{biblatex}
\author{Silvère Nédélec}
\title{Non-Koszulness in a family of properads}
\calclayout

\begin{document}

\ytableausetup{centertableaux, boxsize = 1em}

\begin{abstract} Proving Koszulness of a properad can be very hard, but sometimes one can look at its Koszul complex to look for obstructions for Koszulness. In this paper, we present a method and tools to prove non-Koszulness of many properads in a family of quadratic properads. We illustrate this method on a family of associative and coassociative properads with one quadratic compatibility relation.\end{abstract}

\maketitle

\section{Introduction}

In 2007, B. Vallette extended, in \cite{Va}, Koszul duality from operads (see for example \cite{LoVa}) to properads. This work was already done for other structures encoding bialgebras such as $\frac{1}{2}$-PROPs by M. Markl and A.A. Voronov in \cite{MaVo} and for dioperads by W.L. Gan in \cite{Ga}. In particular, B. Vallette proved that Koszulness of a quadratic properad $\mcal{P}$ is equivalent to acyclicity of its Koszul complex $\mcal{P} \boxtimes \mcal{P}^{\text{\textexclamdown}}$. Koszulness of a properad leads to many results and constructions on this properad, for example the existence of a minimal model of the properad, encoding structures up to homotopy.

In this paper, the goal is to prove non-Koszulness of properads using the Koszul complex. Our main insight is that one does not necessarily need to compute the differential of this complex in order to prove non Koszulness : for given weight and biarity, it suffices that the Euler characteristic of the dg sub-$\bb{S}$-bimodule is not zero to prove non-Koszulness of the properad. We will illustrate this method to a particular case of properads : the properads given by an associative product, a coassociative coproduct and a quadratic compatibility relation depending on four parameters.

Theorem \ref{TheoGen} shows that many properads of this form are not Koszul using only the Euler characteristic condition in weight $4$ and biarity $(2, 4)$, but in the general case it is not enough. In order to get a more general result, one can look at the differential of the Koszul complex, or look at different weights and biarities, but so far we do not know if we can prove Koszulness or non-Koszulness of all properads in this family.

The family we study is the family of properads depending on four parameters $a = (a_1, a_2, a_3, a_4) \in \bb{C}$, the space of generators \[E = \produit{}{}{} \oplus \coproduit{}{}{}\] given by a product and a coproduct without symmetry and the space of relations 

\begin{align*}
R =& \peignegauche{}{}{}{} - \peignedroite{}{}{}{} \oplus \copeignegauche{}{}{}{} - \copeignedroite{}{}{}{} \\
&\oplus \jesus{}{}{}{} - a_1 \dromadroite{}{}{}{} - a_2 \dromagauche{}{}{}{} - a_3 \poissongauche{}{}{}{} - a_4 \poissondroite{}{}{}{}.
\end{align*} We denote by $\between_a$ the term \[\between_a := \jesus{}{}{}{} - a_1 \dromadroite{}{}{}{} - a_2 \dromagauche{}{}{}{} - a_3 \poissongauche{}{}{}{} - a_4 \poissondroite{}{}{}{}\] and we set $\mcal{P}_a := \mcal{F}(E)/(R)$.
	
We first define a rewriting system induced by $\mcal{P}_a$ and find which properads induce a confluent system. This gives a set of relations on $(a_1, a_2, a_3, a_4)$ that characterizes the set of properads that induce a confluent system. We also prove that this condition is equivalent to the existense of an isomorphism between $\mcal{P}_a$ and the composition product $\mcal{A} \boxtimes \mcal{C}$ of the properads encoding associative algebras $\mcal{A}$ and coassociative coalgebras $\mcal{C}$. By \cite[Proposition 8.4]{Va}, this condition implies that $\mcal{P}_a$ is Koszul, thus if one wants to characterize all Koszul properads of this family, one has to consider all properads that do not induce a confluent system.
	
Considering a properad that does not induce a confluent system, we look at its Koszul complex in biarity $(2, 4)$ and weight $4$. This way the complex is given by finite dimensional representations of $\bb{S}_2 \times \bb{S}_4^{\op}$ with equivariant differential maps. Then we look at the isotypic decomposition of this complex and consider their dimensions, that are given by the multiplicities of the Koszul complex. Now if we observe that the Euler characteristics of the isotypic components is not zero, we can say that the properad considered is not Koszul.
	
In order to get multiplicities of the $(\bb{S}_2 \times \bb{S}_4^{\op})$-modules in the Koszul complex, we use three important tools. The first one is a \texttt{Sagemath} script that can be found on \cite{NeWeb} that generates the basis of the free properad on given generators in low weights and a generating family of the ideal generated by given relations in low weights. The second one is a method from \cite{BrMaPe} illustrated by M. Bremner and V. Dotsenko in \cite{BrDo} on a family of operads that we extend to the case of properads and $\bb{S}$-bimodules. The idea is to generate from every couple of partitions $(\lambda, \mu) \vdash (2, 4)$ a matrix whose rank is the multiplicity of $V_{\lambda} \boxtimes V_{\mu}$ in the space of relations. The last tool is the Pieri formula, a special case of the Littlewood--Richardson rule that allows us to compute multiplicities of composition products of $\bb{S}$-bimodules, knowing their multiplicities.

\addtocontents{toc}{\protect\setcounter{tocdepth}{0}}
\section*{Conventions}

Any vector space will be over $\bb{C}$. Every result in this paper also works over any algebraically closed field of characteristic zero.

\subsection*{Permutations, tuples and partitions}

Let $n \in \bb{N}$. Let us denote by $\bb{S}_n$ the permutation group of $\{1, \dots, n\}$. We denote any permutation $\sigma \in \bb{S}_n$ by $[\sigma(1), \dots, \sigma(n)]$, or by cycles with parenthesis, for example $[2, 3, 1, 4] = (1,2,3)$, or $(123)$ if there is no ambiguity.

In this paper, an $n$-tuple will be a family of $n$ natural numbers $\overline{a} = (a_1, \dots, a_n)$. For an $n$-tuple $\overline{a}$, we denote by $\nor{a} := a_1 + \dots + a_n$. 

A partition of $n$ is a non-increasing $k$-tuple $\overline{a} = (a_1, \dots, a_k)$ such that $\nor{a} = n$, this notion is defined up to adding zeros at the end of $\overline{a}$ (the tuples $(3, 2)$ and $(3, 2, 0)$ are the same partitions of $5$). We write $\overline{a} \vdash n$. We will use Young diagrams to represent partitions of natural numbers, for example we have \[(4, 2, 2, 1) = \ydiagram{4, 2, 2, 1}.\] For a partition $\lambda \vdash n$, we denote by $\lambda'$ its conjugate partition of $n$, which is given by the symmetry of the Young diagram of $\lambda$, that is $\lambda_i' = \Card\{j | \lambda_j \geq i \}$.

For $\overline{i}$ an $n$-tuple, we denote by $\bb{S}_{\overline{i}}$ the product of symmetric groups $\bb{S}_{i_1} \times \dots \times \bb{S}_{i_n}$. For $\overline{i}, \overline{j}$ two tuples such that $\nor{i} = \nor{j} = n$, we denote by $\bb{S}_{\overline{i}, \overline{j}}^c$ the set of $\overline{i}, \overline{j}$-connected permutations of $\bb{S}_n$ (see \cite{Va}).

\subsection*{Representations of symmetric groups}

For every group $G$, we denote by $\bb{C}[G]$ the \textit{regular representation} of $G$. As a vector space, $\bb{C}[G]$ has the elements of $G$ as a basis and the left (resp. right) action of $G$ on the basis is given by multiplication on the left (resp. right).
    			
For $H$ a subgroup of $G$ and a left (resp. right) representation $V$ of $G$, we have the structure of a left (resp. right) $H$-module on $V$. We denote this representation by ${}^G_H \!\downarrow V$ (resp. $V \downarrow^G_H$) and call it the \textit{restricted representation} of $V$ on $H$.
    			
Moreover, for a left (resp. right) representation $W$ of $H$, we have the structure of a left (resp. right) representation of $G$ on $\left(\bb{C}[G]\downarrow^G_H\right) \otimes_{H} W$ (resp. $W \otimes_{H} \left({}^G_H\!\downarrow\bb{C}[G]\right)$). We call this representation the \textit{induced representation} of $W$ on $G$ and denote it by $ {}^G_H\!\uparrow W$ (resp. $W\uparrow^G_H$).
    
For $n, m \in \bb{N}$ and two left (resp. right) representations $V$ and $W$ respectively of $\bb{S}_n$ and $\bb{S}_m$, we denote by $V \sqcup_l W$ (resp. $V \sqcup_r W$) the left (resp. right) representation of $\bb{S}_{n + m}$ given by $V \sqcup_l W := {}^{\bb{S}_{n + m}}_{\bb{S}_n \times \bb{S}_m}\!\uparrow (V \otimes W)$ (resp. $V \sqcup_r W := (V \otimes W)\uparrow^{\bb{S}_{n + m}}_{\bb{S}_n \times \bb{S}_m}$).
    			
For a left $\bb{S}_m$-module $V$ and a right $\bb{S}_n$-module $W$, we denote by $V \boxtimes W$ the representation of $\bb{S}_m \times \bb{S}_n^{\op}$ given by the vector space $V \otimes W$ with the left action by $\bb{S}_m$ on $V$ and the right action by $\bb{S}_n$ on $W$.

Let $n \in \bb{N}$. We will denote by $V_{\lambda}$ the irreducible representation of $\bb{S}_n$ corresponding to the partition $\lambda \vdash n$. For a representation $V$ of $\bb{S}_n$, we denote by $m_{\lambda}(V)$ the multiplicity of $V_{\lambda}$ in $V$, and by $[V]_{\lambda} := V_{\lambda}^{\oplus m_{\lambda}(V)}$ the isotypic component of $V$. For more details on representations of symmetric groups, one can see for example \cite{FuHa} or \cite{Sa}.

\addtocontents{toc}{\protect\setcounter{tocdepth}{1}}
\setcounter{tocdepth}{1}
\tableofcontents

\addtocontents{toc}{\protect\setcounter{tocdepth}{0}}
\section*{Aknowledgement}

The author would like to thank gratefully Friedrich Wagemann, Salim Rivière and Johan Leray for all the help, the discussions and the kindness during the PhD thesis and even after.

\addtocontents{toc}{\protect\setcounter{tocdepth}{1}}
\section{Properads}

The notion of properad appeared in order to generalize that of operad (see \cite{LoVa}) to algebraic structures with operations with more than one output. For precise definitions and properties of properads, we refer the reader to \cite{Va}.

\begin{notas} Let us fix several notations.
\begin{enumerate}
\item For $P, Q$ two $\bb{S}$-bimodules, we denote by $Q \boxtimes P$ the connected composition product of $Q$ and $P$ : \[(Q \boxtimes P)(m,n) := \left(\bigoplus_{N, \overline{l}, \overline{k}, \overline{j}, \overline{i}} \bb{C}[\bb{S}_m] \otimes_{S_{\overline{l}}} Q(\overline{l}, \overline{k}) \otimes_{S_{\overline{k}}} \bb{C}[\bb{S}^c_{\overline{k}, \overline{j}}] \otimes_{S_{\overline{j}}} P(\overline{j}, \overline{i}) \otimes_{S_{\overline{i}}} \bb{C}[\bb{S}_n]\right) / \sim,\] where the direct sum runs over the integers $N \in \bb{N}$, the $b$-tuples $\overline{l}$, $\overline{k}$, the $a$-tuples $\overline{j}$, $\overline{i}$ such that $\nor{\overline{l}} = m$, $\nor{\overline{k}} = \nor{\overline{j}} = N$ and $\nor{\overline{i}} = n$, and the equivalence relation~$\sim$ is the following. For $\sigma_1 \in \bb{S}_n$, $\sigma_2 \in S^c_{\overline{k}, \overline{j}}$, $\sigma_3 \in \bb{S}_m$, $(q_1, \dots, q_b) \in Q(\overline{l}, \overline{k})$, $(p_1, \dots, p_a) \in P(\overline{j}, \overline{i})$, $\tau \in \bb{S}_a$ and $\rho \in \bb{S}_b$, we have 

	\begin{align*}
	&\sigma_3 \otimes (q_1, \dots, q_b) \otimes \sigma_2 \otimes (p_1, \dots, p_a) \otimes \sigma_1  \\
	\sim & \sigma_3 \rho_{\overline{l}}^{-1} \otimes (q_{\rho^{-1}(1)}, \dots, q_{\rho^{-1}(b)}) \otimes \rho_{\overline{k}} \sigma_2 \tau_{\overline{j}} \otimes (p_{\tau(1)}, \dots, p_{\tau(a)}) \otimes \tau_{\overline{i}}^{-1}\sigma_1
	\end{align*}
\item For $P, Q$ two weight-graded $\bb{S}$-bimodules (see \cite[Section 2.1]{Va}), that is \[P = \bigoplus_{\rho \in \bb{N}} P^{(\rho)}\text{ and }Q = \bigoplus_{\rho \in \bb{N}} Q^{(\rho)},\] $Q \boxtimes P$ is a weight graded $\bb{S}$-bimodule by taking the sum of weights of composed elements.
\item For $E$ an $\bb{S}$-bimodule, we denote by $\mcal{F}(E)$ the free properad over $E$, which is a weight graded $\bb{S}$-bimodule taking for the weight the number of generators : \[\mcal{F}(E) = \bigoplus_{n \in \bb{N}} \mcal{F}^{(n)}(E).\]
\item For $P, Q$ two weight graded $\mathbb{S}$-bimodules, we denote by \[\underbrace{Q}_{\rho_2} \boxtimes \underbrace{P}_{\rho_1}\] the subspace of $Q \boxtimes P$ generated by the elements of the form $\sigma_3 \otimes (q_1, \dots, q_b) \otimes \sigma_2 \otimes (p_1, \dots, p_b) \otimes \sigma_3$ such that the sum of the weights of $p_1$, \dots, $p_a$ is equal to $\rho_1$ and the sum of the weights of $q_1$, \dots, $q_b$ is equal to $\rho_2$.
\end{enumerate}
\end{notas}

\subsection{Infinitesimal composition product}

\begin{defin} Let $P_1$, $P_2$, $Q_1$ and $Q_2$ be $\bb{S}$-bimodules. We denote by $(Q_1 ; Q_2) \boxtimes (P_1 ; P_2)$ the subspace of $(Q_1 \oplus Q_2) \boxtimes (P_1 \oplus P_2)$ generated by all elements of the form 
				
\begin{equation} \label{FormInfinitesimal} \sigma_3 \otimes (q_1, \dots, q_b) \otimes \sigma_2 \otimes (p_1, \dots, p_a) \otimes \sigma_1 \end{equation} such that exactly one element $p_i$ is in $P_2$ and the others are in $P_1$, and exactly one element $q_i$ is in $Q_2$ and the others are in $Q_1$.
				
More generally, if $P_1, \dots, P_n, Q_1, \dots, Q_m$ are $\bb{S}$-bimodules, let us denote by  \\$(Q_1 ; Q_2, \dots, Q_m) \boxtimes (P_1; P_2, \dots, P_n)$ the subspace of $(Q_1 \oplus Q_2 \oplus \dots \oplus Q_m) \boxtimes (P_1 \oplus P_2 \oplus \dots \oplus P_n)$ generated by all elements of the form (\ref{FormInfinitesimal}) such that for every $2 \leq k \leq n$, there exists a unique $i \in \{1, \dots, a\}$ such that $p_i$ is in $P_k$ and the others are in $P_1$, and for every $2 \leq l \leq n$, there exists a unique $j \in \{1, \dots, b\}$ such that $q_j$ is in $Q_l$, and the others are in $Q_1$
				
Finally, if $\overline{a} = (a_2, \dots, a_n)$ is an $(n-1)$-tuple and $\overline{b} = (b_2, \dots, b_m)$ is an $(m-1)$-tuple, we denote by $(Q_1 ; Q_2, \dots, Q_m) \boxtimes_{\overline{b}, \overline{a}} (P_1 ; P_2, \dots, P_n)$ the subspace of $(Q_1 \oplus Q_2 \oplus \dots \oplus Q_m) \boxtimes (P_1 \oplus P_2 \oplus \dots \oplus P_n)$ generated by all elements of the form (\ref{FormInfinitesimal}) such that for every $2 \leq k \leq n$, there are exactly $a_k$ elements $p_i$ in $P_k$ and the others are in $P_1$, and for every $2 \leq l \leq n$, there are exactly $b_l$ elements $q_i$ in $Q_l$ and the others are in $Q_1$.
\end{defin}
			
\begin{defin} Let $P$ and $Q$ be $\bb{S}$-bimodules. We define the \textit{infinitesimal composition product} of $P$ with $Q$  by \[Q \boxtimes_{(1)} P := (I ; Q) \boxtimes (I ; P).\]
				
More generally, if $P_2, \dots, P_n, Q_2, \dots, Q_m$ are $\bb{S}$-bimodules, we denote the \textit{infinitesimal composition product} of $(P_2, \dots, P_n)$ with $(Q_2, \dots, Q_m)$ by \[(Q_2, \dots, Q_m) \boxtimes_{(1)} (P_2, \dots, P_n) := (I ; Q_2, \dots, Q_m) \boxtimes (I ; P_2, \dots, P_n).\]
				
Finally, if $\overline{a} = (a_2, \dots, a_n)$ is an $(n-1)$-tuple and $\overline{b} = (b_2, \dots, b_m)$ is an $(m-1)$-tuple, we denote by \[(Q_2, \dots, Q_m) \boxtimes_{\overline{b}, \overline{a}} (P_2, \dots, P_n) := (I ; Q_2, \dots, Q_m) \boxtimes_{\overline{b}, \overline{a}} (I ; P_2, \dots, P_n).\]
				
\end{defin}
			
\begin{rem} These definitions are inspired by the analogous definition for $\bb{S}$-modules in \cite[Section 6.1]{LoVa}, this is why we use the same notations. \end{rem}
			
\begin{ex} For example, if $P$ and $Q$ are connected $\bb{S}$-bimodules, we can look at the infinitesimal composition of non trivial blocks of $P$ and $Q$ by looking at the space $\overline{P} \boxtimes_{(1)} \overline{Q}$. \end{ex}
			
\begin{nota} \label{NotaInfinitesimal} If each of the $\bb{S}$-bimodules $P_2, \dots, P_n$ and $Q_2, \dots, Q_m$ is concentrated in one biarity, the infinitesimal composition product $(Q_2, \dots, Q_m) \boxtimes_{\overline{b}, \overline{a}} (P_2, \dots, P_n)$ in a given biarity will be represented with vertices corresponding to the $P_i$'s over vertices corresponding to the $Q_i$'s. For example, if $P_2$ is concentrated in biarity $(2, 3)$, $P_3$ is concentrated in biarity $(3, 2)$ and $Q$ is concentrated in biarity $(4, 3)$, we will represent the infinitesimal composition product $((Q) \boxtimes_{(2), (1, 1)} (P_2, P_3))(8, 6)$ by Figure  \ref{RepInfinitesimal}. The rectangles represent the formal operations, the horizontal line in the middle corresponds to any connected permutation, not to be confused with the notation from \cite[Definition 19]{Ma} which does not make sense here, and the vertical line corresponds to $I$.
			
\begin{figure}[h!]
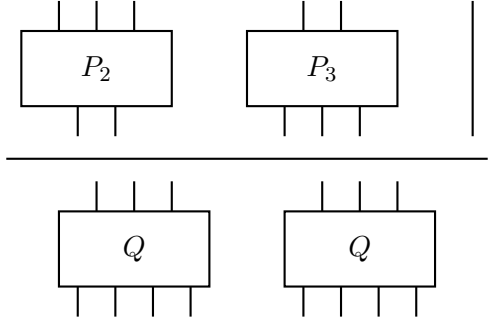
 
\centering
\RepInfinitesimal
\caption{Representation of an infinitesimal composition product} 
\label{RepInfinitesimal} 
\end{figure}
			
\end{nota}
			
With these definitions, we state a very useful theorem which describes how to compose two formal operations defined by representations of symmetric groups. Because irreducible representations of product of symmetric groups are tensor products of irreducible representations of the symmetric groups, it is enough to consider tensor products of representations.
		
\begin{theo}\label{TheoComposBlocks} Let $k, l, m, n \in \bb{N}^*$ and 
		
\begin{enumerate}
\item $V$ be a left $\bb{S}_k$-module,
\item $W$ be a right $\bb{S}_l$-module,
\item $X$ be a left $\bb{S}_m$-module,
\item $Y$ be a right $\bb{S}_n$-module.
\end{enumerate}
		
Take $\mcal{Q} := V \boxtimes W$ and $\mcal{P} := X \boxtimes Y$. We have the isomorphism of $\bb{S}$-bimodules \[(\mcal{Q} \boxtimes_{(1)} \mcal{P})(k + m - 1, l + n - 1) \simeq (V \sqcup_l \ ^{\bb{S}_m}_{\bb{S}_{m - 1}}\downarrow X) \boxtimes (W\downarrow^{\bb{S}_l}_{\bb{S}_{l - 1}} \sqcup_r Y).\]
\end{theo}
		
\begin{rem} Here we take the specific biarity where exactly one edge will connect the unique non trivial operation of $\mcal{P}$ with the unique non trivial operation of $\mcal{Q}$. This is a specific type of dioperadic composition, see \cite[Section 1]{Ga}. We use Notation \ref{NotaInfinitesimal} in Figure \ref{CompBlocksDrawing}.
\end{rem} 
		
\begin{figure}[h!]
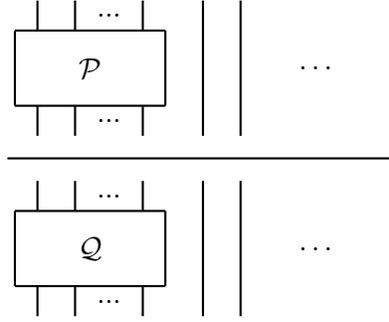
 \caption{Graphic representation of the composition of two operations} \label{CompBlocksDrawing}
\centering
\CompBlocksDrawing
\end{figure}
		
\begin{proof}
		
Let us denote by $A$ the left side of Theorem \ref{TheoComposBlocks}. We have, by definition,
		
\begin{align*}
A = &\left(\bigoplus_{\overline{k}, \overline{l}, \overline{m}, \overline{n}} \bb{C}[\bb{S}_{m + k - 1}] \otimes_{\bb{S}_{\overline{k}}} (\bb{C} \otimes \dots \otimes \mcal{Q}(k, l) \otimes \dots \otimes \bb{C}) \right. \\
&\left. \otimes_{\bb{S}_{\overline{l}}} \bb{C}[\bb{S}^c_{\overline{l}, \overline{m}}] \otimes_{\bb{S}_{\overline{m}}} (\bb{C} \otimes \dots \otimes \mcal{P}(m, n) \otimes \dots \otimes \bb{C}) \otimes_{\bb{S}_{\overline{n}}} \bb{C}[\bb{S}_{n + l - 1}]\right)/\sim 
\end{align*} 
the sum over tuples $\overline{k} = (1, \dots, k, \dots, 1)$, $\overline{l} = (1, \dots, l, \dots, 1)$, $\overline{m} = (1, \dots, m, \dots, 1)$, $\overline{n} = (1, \dots, n, \dots, 1)$, with $\nor{\overline{k}} = m + k - 1$, $\nor{\overline{l}} = \nor{\overline{m}} = l + m - 1$ and $\nor{\overline{n}} = n + l - 1$. But because of the relation $\sim$, we can take only the terms of the sum where $\mcal{Q}(k, l)$ and $\mcal{P}(m, n)$ appear in first place, thus removing the tensor products by $\bb{C}$, we find
	
\[A \simeq \left(\bb{C}[\bb{S}_{m + k - 1}] \otimes_{\bb{S}_k} \mcal{Q}(k, l) \otimes_{\bb{S}_l} \bb{C}[\bb{S}^c_{(l, 1, \dots, 1), (m, 1, \dots, 1)}] \otimes_{\bb{S}_m} \mcal{P}(m, n) \otimes_{\bb{S}_n} \bb{C}[\bb{S}_{n + l - 1}]\right)/\sim',\] where the relation $\sim'$ is given by \[\sigma_1 \otimes q \otimes \sigma_2 \otimes p \otimes \sigma_3 \sim' \sigma_1 (\id_k, \rho^{-1}) \otimes q \otimes (\id_l, \rho)\sigma_2(\id_m, \tau) \otimes p \otimes (\id_n, \tau^{-1})\sigma_3.\] for $\rho \in \bb{S}_{m - 1}$, $\tau \in \bb{S}_{l - 1}$, and where we denoted for $\sigma \in \bb{S}_n$ and $\sigma' \in \bb{S}_m$ by $(\sigma, \sigma')$ the corresponding permutation in $\bb{S}_{n + m}$. The first goal of this proof is to understand the bimodule $\bb{C}[\bb{S}^c_{(l, 1, \dots, 1), (m, 1, \dots, 1)}]$ as a left $\bb{S}_l$-module and a right $\bb{S}_m$-module. We have \[\bb{S}^c_{(l, 1, \dots, 1),(m, 1, \dots, 1)} \simeq \{1, \dots, l\} \times \bb{S}_{l - 1} \times \{1, \dots m\} \times \bb{S}_{m - 1}.\] Thus \[\bb{C}[\bb{S}^c_{(l, 1, \dots, 1),(m, 1, \dots, 1)}] \simeq \bb{C}^l \otimes \bb{C}[\bb{S}_{l - 1}] \otimes \bb{C}^m \otimes \bb{C}[\bb{S}_{m - 1}].\]  Let us first define two maps, for $i \in \{1, \dots, l\}$ :

\Fonction{\delta_i}{\{1, \dots, l - 1\}}{\{1, \dots, l\}}{a}{\begin{cases} a &\text{ if } a < i \\ a + 1 &\text{ if }a \geq i\end{cases}} and

\Fonction{s_i}{\{1, \dots, l\}}{\{1, \dots, l - 1\}}{a}{\begin{cases} a &\text{ if } a < i \\ a - 1 &\text{ if }a \geq i\end{cases}.} The left action of $\bb{S}_l$ on $\bb{C}^l \otimes \bb{C}[\bb{S}_{l - 1}]$ is given, for $(e_i)$ the canonical basis of $\bb{C}^l$, $i \in \{1, \dots, l\}$, $\sigma \in \bb{S}_l$ and $\rho \in \bb{S}_{l - 1}$, by \[\sigma \cdot (e_i \otimes \rho) = e_{\sigma(i)} \otimes s_{\sigma(i)} \sigma \delta_i \rho.\] The structure on $\bb{C}^m \otimes \bb{C}[\bb{S}_{m - 1}]$ on the right is given by a similar formula. Let us now define another structure on $\bb{C}^l \otimes \bb{C}[\bb{S}_{l - 1}]$ given by an isomorphism with $\bb{C}[\bb{S}_l]$. We define again two maps.
		
\Fonction{u}{\bb{S}_{l - 1}}{\bb{S}_l}{\sigma}{[\sigma(1), \dots, \sigma(l - 1), l]} and
		
\Fonction{d}{\{\sigma \in \bb{S}_l\text{ such that }\sigma(l) = l\}}{\bb{S}_{l - 1}}{\sigma}{[\sigma(1), \dots, \sigma(l - 1)].} We have a morphism of vector spaces 
		
\Fonction{\varphi}{\bb{C}^l \otimes \bb{C}[\bb{S}_{l - 1}]}{\bb{C}[\bb{S}_l]}{e_i \otimes \rho}{(i, i + 1, \dots, l)u(\rho),} which is an isomorphism, with \Fonction{\varphi^{-1}}{\bb{C}[\bb{S}_l]}{\bb{C}^l \otimes \bb{C}[\bb{S}_{l - 1}]}{\sigma}{e_{\sigma(l)} \otimes d\left((l, \dots, \sigma(l)) \sigma\right).} This isomorphism induces a left $\bb{S}_l$-module structure on $\bb{C}^l \otimes \bb{C}[\bb{S}_{l - 1}]$, given by 
		
\begin{align*}
\sigma \star (e_i \otimes \rho) &= \varphi^{-1}(\sigma \varphi(e_i \otimes \rho)) \\
&= \varphi^{-1}(\sigma (i, \dots, l) u(\rho)) \\
&= e_{\sigma(i)} \otimes d((l, \dots, \sigma(i))\sigma(i, \dots, l) u(\rho)) \\
&= \sigma \cdot (e_i \otimes \rho).
\end{align*} We can do the same computation for the right action of $\bb{S}_m$ and prove that, as $\bb{S}$-bimodules, we have \[\bb{C}[\bb{S}^c_{(l, 1, \dots, 1),(m, 1, \dots, 1)}] \simeq \bb{C}[\bb{S}_l] \boxtimes \bb{C}[\bb{S}_m].\]
		
Now $A$ is isomorphic to {\scriptsize \[\left(\bb{C}[\bb{S}_{m + k - 1}] \otimes_{\bb{S}_k} V \otimes W \otimes_{\bb{S}_l} (\bb{C}^l \otimes \bb{C}[\bb{S}_{l - 1}]) \otimes (\bb{C}^m \otimes \bb{C}[\bb{S}_{m - 1}]) \otimes_{\bb{S}_m} X \otimes Y \otimes_{\bb{S}_n} \bb{C}[\bb{S}_{n + l - 1}]\right)/\sim',\]} but the relation $\sim'$ can be cut in half saying that  $A$ is isomorphic to
	
{\scriptsize \[\left(\bb{C}[\bb{S}_{m + k - 1}] \otimes_{\bb{S}_k} V \otimes (\bb{C}^m \otimes \bb{C}[\bb{S}_{m - 1}]) \otimes_{\bb{S}_m} X\right)/\sim'_1 \otimes \left(W \otimes_{\bb{S}_l} (\bb{C}^l \otimes \bb{C}[\bb{S}_{l - 1}]) \otimes Y \otimes_{\bb{S}_n} \bb{C}[\bb{S}_{n + l - 1}]\right)/\sim'_2.\]}
	
Now let us study the first half 
		
\begin{align*}
B &:= \left(\bb{C}[\bb{S}_{m + k - 1}] \otimes_{\bb{S}_k} V \otimes (\bb{C}^m \otimes \bb{C}[\bb{S}_{m - 1}]) \otimes_{\bb{S}_m} X\right)/\sim'_1\\
&\simeq \bb{C}[\bb{S}_{m + k - 1}] \otimes_{\bb{S}_k \times \bb{S}_{m - 1}} (V \otimes ((\bb{C}^m \otimes \bb{C}[\bb{S}_{m - 1}]) \otimes_{\bb{S}_m} X)).
\end{align*} The idea now is to understand the left action of $\bb{S}_{m - 1}$ on $((\bb{C}^m \otimes \bb{C}[\bb{S}_{m - 1}]) \otimes_{\bb{S}_m} X)$. The action of $\bb{S}_{m-1}$ on $\bb{C}^m \otimes \bb{C}[\bb{S}_{m - 1}]$ is just the one on $\bb{C}[\bb{S}_{m - 1}]$. We have the sequence of isomorphisms \[((\bb{C}^m \otimes \bb{C}[\bb{S}_{m - 1}]) \otimes_{\bb{S}_m} X) \simeq \bb{C}[\bb{S}_m] \otimes_{\bb{S}_m} X \simeq X.\] The action of $\bb{S}_{m-1}$ on $\bb{C}[\bb{S}_m]$ in the second term is given by $\varphi$ : let $\rho \in \bb{S}_{m}$ and $\sigma \in \bb{S}_{m - 1}$, we have 
	
\begin{align*}
\sigma \cdot \rho &= \varphi(\sigma \varphi^{-1}(\rho)) \\
&= \varphi(e_{\rho(m)} \otimes \sigma (m, \dots \rho(m)) \rho) \\
&= (\rho(m), \dots, m) u(\sigma d( (m, \dots, \rho(m)) \rho)).
\end{align*} Thus we have $\sigma \cdot \id = u(\sigma)$, and the action of $\sigma \in \bb{S}_{m - 1}$ on $X$ is the action on ${}^{\bb{S}_m}_{\bb{S}_{m - 1}}\!\downarrow X$. Finally, we get \[B \simeq \bb{C}[\bb{S}_{m + k - 1}] \otimes_{\bb{S}_k \times \bb{S}_{m - 1}} (V \otimes {}^{\bb{S}_m}_{\bb{S}_{m - 1}}\!\downarrow X) = V \sqcup_l {}^{\bb{S}_m}_{\bb{S}_{m - 1}}\!\downarrow X.\] We do the same for the second half and get Theorem \ref{TheoComposBlocks}.
		
\end{proof}
		
A special case of this theorem is when $V, W, X$ and $Y$ are regular representations of respective symmetric groups.
		
\begin{coro} \label{CoroComposBlocks} If $V = \bb{C}[\bb{S}_k]$, $W = \bb{C}[\bb{S}_l]$, $X = \bb{C}[\bb{S}_m]$ and $Y = \bb{C}[\bb{S}_n]$, we have \[(\mcal{Q} \boxtimes_{(1)} \mcal{P})(m + k - 1, n + l - 1) = \bb{C}[\bb{S}_{m + k - 1} \times \bb{S}_{n + l - 1}^{\op}]^{\oplus lm}\]
\end{coro}
		
\begin{proof} This result is given by the fact that, as left $\bb{S}_l$-modules, \[{}^{\bb{S}_l}_{\bb{S}_{l - 1}}\!\downarrow \bb{C}[\bb{S}_l] \simeq \bb{C}[\bb{S}_{l - 1}]^{\oplus l}.\] and as right $\bb{S}_m$-modules, we have \[\bb{C}[\bb{S}_m]\downarrow^{\bb{S}_m}_{\bb{S}_{m - 1}} \simeq \bb{C}[\bb{S}_{m - 1}]^{\oplus m}\] A proof of this property is given in the proof of Theorem \ref{TheoComposBlocks}.
\end{proof}

\subsection{Koszul duality for properads}

Here we just give an idea of Koszul duality for properads. For a precise presentation, see \cite{Va}. The idea is the following :

\begin{itemize}
\item One can define coproperads as comonoids in the monoidal category $(\bb{S}-\Bimod, \boxtimes, I)$.
\item One can define the differential graded versions of $\bb{S}$-bimodules, properads and coproperads : dg $\bb{S}$-bimodules, dg properads and dg coproperads.
\item To any quadratic properad $\mcal{P}$, that is $\mcal{P} = \mcal{F}(E)/(R)$ with $R \subset \mcal{F}^{(2)}(E)$, one can define its Koszul dual dg-coproperad $\mcal{P}^{\text{\textexclamdown}}$.
\item We say that a quadratic properad $\mcal{P}$ is Koszul if the complex $\mcal{P} \boxtimes \mcal{P}^{\text{\textexclamdown}}$ is acyclic.
\item For a quadratic properad $\mcal{P} := \mcal{F}(E)/(R)$, one can define the Koszul dual properad of $\mcal{P}$ by generators and relations :  $\mcal{P}^{\text{!}} := \mcal{F}(E)/(R^{\perp})$
\end{itemize}

\begin{rem} We do not give the differential for the complex $\mcal{P} \boxtimes \mcal{P}^{\text{\textexclamdown}}$ in this paper because we will only need dimensions and multiplicities of the chains in order to prove non Koszulness.
\end{rem}

We also remind an important tool to prove Koszulness of properads. Let $V$ and $W$ be two $\bb{S}$-bimodules, $R \subset \mcal{F}^{(2)}(V)$, $S \subset \mcal{F}^{(2)}(W)$ and \[D \subset (V \boxtimes_{(1)} W) \oplus (W \boxtimes_{(1)} V).\] We denote by $\mcal{A}$ and $\mcal{B}$ the properads given respectively by $\mcal{F}(V)/(R)$ and $\mcal{F}(W)/(S)$, and set \[\mcal{P} := \mcal{F}(V \oplus W)/(R \oplus D_{\lambda} \oplus S).\]

\begin{defin}[Replacement rule, {\cite[Section 8]{Va}}] Let $\lambda$ be a morphism of $\bb{S}$-bimodules \[\lambda : W \boxtimes_{(1)} V \rightarrow V \boxtimes_{(1)} W\] such that $D$ is the image of \[(\id, -\lambda) : W \boxtimes_{(1)} V \rightarrow (W \boxtimes_{(1)} V) \oplus (V \boxtimes_{(1)} W).\]
				
If the two morphisms \[\underbrace{\mcal{A}}_{1} \boxtimes \underbrace{\mcal{B}}_{2} \rightarrow \mcal{P} \text{ and } \underbrace{\mcal{A}}_{2} \boxtimes \underbrace{\mcal{B}}_{1} \rightarrow \mcal{P}\] are injective, we call $\lambda$ a \textit{replacement rule} and we denote $D$ by $D_{\lambda}$.
\end{defin}
				
\begin{theo}[{\cite[Proposition 8.4]{Va}}] \label{TheoReplacement} Let $\mcal{P}$ be a properad of the form \[\mcal{P} = \mcal{F}(V \oplus W)/(R \oplus D_{\lambda} \oplus S)\] with $R \subset \mcal{F}^{(2)}(V)$, $S \subset \mcal{F}^{(2)}(W)$ and $\lambda$ a replacement rule, and such that $\sum_{n, m} \dim((V \oplus W)(n, m))$ is finite and $\mcal{A} := \mcal{F}(V)/(R)$ and $\mcal{B} := \mcal{F}(W)/(S)$ are Koszul properads. Then $\mcal{P}$ is Koszul.
\end{theo}
				
This theorem proves Koszulness of a properad, based on Koszulness of two other properads. Most of the time, we use this theorem in the case where $V$ is  given by operations (with one output) and $W$ is given by cooperations (with one input), thus $\mcal{A}$ and $\mcal{B}^{\op}$ are operads, where $\mcal{B}^{\op}$ is the reversed $\bb{S}$-bimodule of $\mcal{B}$ (see \cite[Section 8]{Va}), and Koszulness of $\mcal{A}$ and $\mcal{B}$ can be proved using operadic tools such as rewriting theory. Moreover, proving that $\lambda$ is a replacement rule can be done computing dimensions of modules.

Using the Koszul dual properad $\mcal{P}^{\text{!}}$, one can find multiplicities of the Koszul dual coproperad $\mcal{P}^{\text{\textexclamdown}}$.

\begin{prop} \label{PropDual} If we have the decomposition into isotypic components of \[\mcal{P}^{\text{!}}(m, n) = \sum_{(\lambda, \mu) \vdash (m, n)} V_{\lambda, \mu}^{\oplus m_{\lambda, \mu}(\mcal{P}^{\text{!}}(m, n))},\] then we have \[\mcal{P}^{\text{\textexclamdown}}(m, n) = \sum_{(\lambda, \mu) \vdash (m, n)} V_{\lambda, \mu}^{\oplus m_{\lambda', \mu'}(\mcal{P}^{\text{!}}(m, n))}.\] In other words, $m_{\lambda, \mu}(\mcal{P}^{\text{!}}(m, n)) = m_{\lambda', \mu'}(\mcal{P}^{\text{\textexclamdown}}(m, n))$. 
\end{prop}

\section{Representation theory tools for properads}

In this section we present the two main representation theoretic tools : Pieri's rules and representation matrices.

\subsection{The Pieri rules}
    		
Here we remind general results on representations. Note that we can state these results for left or right modules, but here we consider only left modules. Let $V$ and $W$ be respectively representations of $\bb{S}_n$ and $\bb{S}_m$. The following rule is very useful if one wants to understand the representation structure of $V \sqcup_l W$. If we know the isotypic decompositions of $V$ and $W$, the Littlewood--Richardson numbers give us the decomposition of $V \sqcup_l W$. The proof of this rule can be found in \cite[Theorem 4.9.4]{Sa}.
    		
\begin{defin}[See {\cite[Appendix A]{FuHa}} and {\cite[Theorem 4.9.4]{Sa}}] Let $\lambda \vdash n$, $\mu = (\mu_1, \mu_2, \dots, \mu_k) \vdash m$ and $\nu \vdash m + n$ be three partitions. We call Littlewood--Richardson tableau of type $(\lambda, \mu ; \nu)$ any Young diagram of shape $\lambda$, completed into a Young diagram of shape $\nu$ by boxes filled by $\mu_1$ $1$s, $\mu_2$ $2$s, $\dots$, $\mu_k$ $k$s such that each row is non decreasing and each column is strictly increasing, and such that if one takes the sequence formed by these numbers listed from right to left starting from the first row to the last one, one has at any point more $p$s than $p+1$s for $1 \leq p \leq k - 1$.
\end{defin}
    			
\begin{theo}[Littlewood--Richardson's rule] For $n, m \in \bb{N}$, $\lambda \vdash n$ and $\mu \vdash m$, we have \[V_{\lambda} \sqcup_l V_{\mu} = \bigoplus_{\nu \vdash n + m} V_{\nu}^{\oplus N_{\lambda\mu}^{\nu}},\] where the $N_{\lambda\mu}^{\nu}$ are the Littlewood--Richardson numbers, which are the number of Littlewood--Richardson tableaux of type $(\lambda, \mu ; \nu)$.
\end{theo}
    			
\begin{ex} Littlewood--Richardson's rule allows us for example to compute \[V_{(2, 1)} \sqcup_l V_{(2, 1)} = V_{(4, 2)} \oplus V_{(4, 1, 1)} \oplus V_{(3, 3)} \oplus V_{(3, 2, 1)}^{\oplus 2} \oplus V_{(3, 1, 1, 1)} \oplus V_{(2, 2, 2)} \oplus V_{(2, 2, 1, 1)}.\]
    			
In fact, the corresponding Littlewood--Richardson tableaux are \[\begin{ytableau} \  & \  & 1 & 1 \\ \  & 2 \end{ytableau}\ , \ \ \begin{ytableau} \  & \  & 1 & 1 \\ \  \\ 2 \end{ytableau}\ , \ \ \begin{ytableau} \  & \  & 1 \\ \  & 1 & 2 \end{ytableau}\ , \ \ \begin{ytableau} \  & \  & 1 \\ \  & 2 \\ 1 \end{ytableau}\ , \ \ \begin{ytableau} \  & \  & 1 \\ \  & 1 \\ 2 \end{ytableau}\ , \ \ \begin{ytableau} \  & \  & 1 \\ \  \\ 1 \\ 2 \end{ytableau}\ , \ \ \begin{ytableau} \  & \  \\ \  & 1 \\ 1 & 2 \end{ytableau}\ \text{ and }\ \begin{ytableau} \  & \  \\ \  & 1 \\ 1 \\ 2 \end{ytableau}\] \end{ex}
    		
Consequences of the Littlewood--Richardson rule are Pieri's rule and its dual form.
    		
\begin{theo}[Pieri's rules] For $n, m \in \bb{N}$ and $\lambda \vdash n$, we have \[V_{\lambda} \sqcup_l V_{(m)} = \bigoplus_{\substack{\mu \vdash n + m \\ \mu \geq^c \lambda}} V_{\mu},\] the sum over the partitions $\mu \vdash n + m$ which can be obtained from $\lambda$ by adding boxes to its Young diagram, maximum one by column, this explains the notation $\geq^c$. We also have \[V_{\lambda} \sqcup_l V_{(1^m)} = \bigoplus_{\substack{\mu \vdash n + m \\ \mu \geq^l \lambda}} V_{\mu},\] the sum over the partitions $\nu \vdash n + m$ which can be obtained from $\lambda$ by adding boxes to its Young diagram, maximum one by row, this explains the notation $\geq^l$.
\end{theo}
    			
From these rules one can deduce two corollaries, the first one allows us to compute the isotypic decomposition of any $\bb{S}_n$-module as an $\bb{S}_{n+1}$-module.
    		
\begin{coro} \label{coroPieri1} For $n \in \bb{N}$, $\lambda \vdash n$ and $\mu \vdash n + 1$, we have \[m_{\mu}({}_{\bb{S}_n}^{\bb{S}_{n + 1}}\!\uparrow(V_{\lambda})) = \begin{cases} 1 &\text{ if } \mu \geq \lambda \\ 0 &\text{ else.} \end{cases}\]
    		
\end{coro}
    			
\begin{proof} Pieri's rule gives us \[{}_{\bb{S}_n}^{\bb{S}_{n + 1}}\!\uparrow V_{\lambda} = {}_{\bb{S}_n}^{\bb{S}_{n + 1}}\!\uparrow(V_{\lambda} \otimes \bb{C}) = V_{\lambda} \sqcup_l V_{(1)} = \bigoplus_{\substack{\mu \vdash n + 1 \\ \mu \geq^c \lambda}} V_{\mu}\] the sum over partitions $\mu \vdash n + 1$ which can be obtained from $\lambda$ by adding one box, that is $\mu \geq \lambda$.
    			
\end{proof}
    			
And the second one allows us to compute any $\bb{S}_n$-module as an $\bb{S}_{n-1}$-module.
    		
\begin{coro} \label{coroPieri2} For $n \in \bb{N}$, $\lambda \vdash n$ and $\mu \vdash n - 1$, we have \[m_{\mu}({}^{\bb{S}_n}_{\bb{S}_{n - 1}}\!\downarrow(V_{\lambda})) = \begin{cases} 1 &\text{ if } \mu \leq \lambda \\ 0 &\text{ else.} \end{cases}\]
    		
\end{coro}
    			
\begin{proof} This result comes from Frobenius reciprocity, see \cite[Corollary 3.20]{FuHa}, and Corollary \ref{coroPieri1}. One consequence of Frobenius reciprocity is, for $\lambda \vdash n$ and $\mu \vdash n - 1$,
    			
\[m_{\mu}({}^{\bb{S}_n}_{\bb{S}_{n - 1}}\!\downarrow(V_{\lambda})) = m_{\lambda}({}^{\bb{S}_n}_{\bb{S}_{n - 1}}\!\uparrow(V_{\mu})) = \begin{cases} 1 &\text{ if } \mu \leq \lambda \\ 0 &\text{ else} \end{cases}\]
    			
\end{proof}
    		
\begin{exs} For $\lambda$ given by the Young diagram \[\ydiagram{2, 2},\] and $\mu_1$, $\mu_2$ and $\mu_3$ given respectively by the diagrams \[\ydiagram{3}\ ,\ \  \ydiagram{2, 1}\ \text{ and }\ \ydiagram{4, 1},\] we have \[m_{\mu_1}((V_{\lambda})\downarrow^{\bb{S}_4}_{\bb{S}_3}) = 0, m_{\mu_2}((V_{\lambda})\downarrow^{\bb{S}_4}_{\bb{S}_3}) = 1\text{ and }m_{\mu_3}((V_{\lambda})\uparrow^{\bb{S}_5}_{\bb{S}_4}) = 0\]
\end{exs}

\subsection{Representation matrices}

\label{SecRepmat}

Here we describe a method presented by M.Bremner and V.Dotsenko in \cite{BrDo} to compute multiplicities of operads by generators and relations knowing a basis of the free properad over the generators and a generating family of the space of relations, both in a given arity. Then we generalize this method to properads given by generators and relations.

\subsubsection{The case of $\bb{S}$-modules}

Let $n \in \bb{N}$ and $d_{\lambda}$ be the dimension of the irreducible representation $V_{\lambda}$ for every $\lambda \vdash n$. We have an isomorphism of vector spaces \[\varphi : \bb{C}[\bb{S}_n] \rightarrow \bigoplus_{\lambda \vdash n} \mcal{M}_{d_{\lambda}}(\bb{C}),\] where $\mcal{M}_m(\bb{C})$ is the space of matrices of dimension $m \times m$ over $\bb{C}$. This isomorphism induces the structure of a representation of $\bb{S}_n$ on $\bigoplus_{\lambda \vdash n} \mcal{M}_{d_{\lambda}}(\bb{C})$. See for example \cite[Part 1]{BrMaPe} for the construction of this morphism. For every $\lambda \vdash n$, we denote by $\varphi_{\lambda}$ the composition of $\varphi$ with the projection onto $\mcal{M}_{d_{\lambda}}(\bb{C})$, and we have that $\varphi_{\lambda}$ restricts to an isomorphism from $V_{\lambda}^{\oplus d_{\lambda}}$ to $\mcal{M}_{d_{\lambda}}(\bb{C})$. We denote this isomorphism by $P_{\lambda}$.
    			
\begin{prop} \label{propCliftonoperad} Let $R(n)$ be the $\bb{S}_n$-module generated by $x_1, \dots, x_p$ in $\bb{C}[\bb{S}_n]^{\oplus q}$, and for every $1 \leq i \leq p$, $x_i = (x_i^{(1)}, \dots, x^{(q)}_i)$ their decompositions according to the direct sum $\bb{C}[\bb{S}_{n}]^{\oplus q}$. Then we have, for every $\lambda \vdash n$, \[m_{\lambda}(R(n)) = \rk \left( \left(P_{\lambda}(x^{(j)}_i)\right)_{\substack{1 \leq i \leq p \\ 1 \leq j \leq q}}\right)\]
    			
Moreover, we call the block matrix $\left(P_{\lambda}(x^{(j)}_i)\right)_{\substack{1 \leq i \leq p \\ 1 \leq j \leq q}}$ the representation matrix of $R(n)$ for $\lambda$.
\end{prop}

\begin{ex} Let $\mcal{P}$ be an operad given by generators $E$ and relations $R$ such that, for given $n, q \in \bb{N}$, the underlying $\bb{S}_n$-module of $\mcal{F}(E)(n)$ is isomorphic to $q$ copies of the regular representation, that is \[\mcal{F}(E)(n) = \bb{C}[\bb{S}_n]^{\oplus q} = \bigoplus_{\lambda \vdash n} V_{\lambda}^{\oplus qd_{\lambda}},\] and such that $R(n)$ is generated, as an $\bb{S}_n$-module, by a finite number of elements in $\mcal{F}(E)(n)$, say $x_1, \dots, x_p$. In that case, Proposition \ref{propCliftonoperad} gives us the isotypic decomposition of $R(n)$ computing the ranks of matrices.
\end{ex}

\begin{rem} The advantage of dealing with these block matrices is that they are smaller that the matrix representing the whole $R(n)$, and their ranks are therefore easier to compute. 
\end{rem}
    			
To prove this proposition, we will need a lemma.
    			
\begin{lemme} \label{lemma1Clifton} Let $n \in \bb{N}$, $A_{i, j} \in \mcal{M}_n(\bb{C}) =: \mcal{M}$ for every $1 \leq i \leq p$ and $1 \leq j \leq q$. We have \[\dim(\mcal{M} \cdot (A_{1, 1}, \dots, A_{1, q}) + \dots + \mcal{M} \cdot (A_{p, 1}, \dots, A_{p, q})) = \rk\left( \left( A_{i, j}\right)_{\substack{1 \leq i \leq p \\ 1 \leq j \leq q}} \right) \cdot n,\] where $\mcal{M}\cdot (A_{i, 1}, \dots, A_{i, q}) := \spn\{(NA_{i, 1}, \dots, NA_{i, q}), N \in \mcal{M}\}$.
\end{lemme}
    			
\begin{proof} We denote by $V$ the space \[V := \mcal{M} \cdot (A_{1, 1}, \dots, A_{1, q}) + \dots + \mcal{M} \cdot (A_{p, 1}, \dots, A_{p, q}).\] Let us first remark that we have an isomorphism 
    			
\begin{align*}
\mcal{M}^{\oplus q} &\simeq \mcal{M}_{n, qn}(\bb{C}) \\
(A_1, \dots, A_q) &\mapsto (A_1 | \cdots | A_q).
\end{align*}
    			
Let us fix $1 \leq i \leq p$, we have \[\mcal{M} \cdot (A_{i, 1}, \dots, A_{i, q}) \simeq \spn\{(NA_{i, 1} | \cdots | NA_{i, q}), N \in \mcal{M}\}.\] We now denote for $1 \leq j \leq q$ the rows of $A_{i, j}$ by $L_{1, j}, \dots, L_{n, j}$ : \[A_{i, j} = \begin{pmatrix} L_{1, j} \\ \hline \vdots \\ \hline L_{n, j} \end{pmatrix}.\] Thus the vector space $\mcal{M} \cdot (A_{i, 1}, \dots, A_{i, q})$ is isomorphic to \[\spn\left\{ \begin{pmatrix} \begin{matrix} N_{1, 1}L_{1, 1} + \dots + N_{1, n}L_{n, 1} \\ \hline \vdots \\ \hline N_{n, 1}L_{1, 1} + \dots + N_{n, n}L_{n, 1} \end{matrix} & \begin{vmatrix} \\ \\ \cdots \\ \\ \\ \end{vmatrix} & \begin{matrix} N_{1, 1}L_{1, q} + \dots + N_{1, n}L_{n, q} \\ \hline \vdots \\ \hline N_{n, 1}L_{1, q} + \dots + N_{n, n}L_{n, q} \end{matrix} \end{pmatrix}, N \in \mcal{M}\right\}.\] If we denote the rows of $(A_{i, 1} | \cdots | A_{i, q})$ by $L_1^{(i)}, \dots, L_n^{(i)}$ : \[(A_{i, 1} | \cdots | A_{i, q}) = \begin{pmatrix} L_1^{(i)} \\ \hline \vdots \\ \hline L_n^{(i)} \end{pmatrix},\] we get 
    			
\begin{align*}
\mcal{M} \cdot (A_{i, 1}, \dots, A_{i, q}) &\simeq \spn\left\{ \begin{pmatrix} N_{1, 1}L_1^{(i)} + \dots + N_{1, n}L_n^{(i)} \\ \hline \vdots \\ \hline N_{n, 1}L_1^{(i)} + \dots + N_{n, n}L_n^{(i)} \end{pmatrix}, N \in \mcal{M}\right\} \\
&= \begin{pmatrix} \spn(L_1^{(i)}, \dots, L_n^{(i)}) \\ \hline \vdots \\ \hline \spn(L_1^{(i)}, \dots, L_n^{(i)}) \end{pmatrix}.
\end{align*} Finally, we have 
    			
\begin{align*}
V &= \begin{pmatrix} \spn(L_1^{(1)}, \dots, L_n^{(1)}) \\ \hline \vdots \\ \hline \spn(L_1^{(1)}, \dots, L_n^{(1)}) \end{pmatrix} + \dots + \begin{pmatrix} \spn(L_1^{(p)}, \dots, L_n^{(p)}) \\ \hline \vdots \\ \hline \spn(L_1^{(p)}, \dots, L_n^{(p)}) \end{pmatrix} \\
&= \begin{pmatrix} \spn(L_1^{(1)}, \dots, L_n^{(1)}, \dots, L_1^{(p)}, \dots, L_n^{(p)}) \\ \hline \vdots \\ \hline \spn(L_1^{(1)}, \dots, L_n^{(1)}, \dots, L_1^{(p)}, \dots, L_n^{(p)}) \end{pmatrix}.
\end{align*} But the dimension of the space $\spn(L_1^{(1)}, \dots, L_n^{(1)}, \dots, L_1^{(p)}, \dots, L_n^{(p)})$ is $\rk\left( \left( A_{i, j}\right)_{\substack{1 \leq i \leq p \\ 1 \leq j \leq q}} \right)$, which concludes the proof.
\end{proof}
    			
\begin{proof}[Proof of Proposition \ref{propCliftonoperad}] We have the following commutative diagram \[\xymatrix{R(n) \ar@{^{(}->}[r] & \bb{C}[\bb{S}_n]^{\oplus q} \ar[r]^{\varphi_{\lambda}^{\oplus q}} \ar@<-2pt>[d]_p & M_{d_{\lambda}}(\bb{C})^{\oplus q} \\ & V_{\lambda}^{\oplus qd_{\lambda}} \ar@<-2pt>[u]_i \ar[ur]^{P_{\lambda}^{\oplus q}} & }\] with $p$ and $i$ respectively the projection and inclusion of $V_{\lambda}^{\oplus qd_{\lambda}}$ in $\bb{C}[\bb{S}_n]^{\oplus q}$. Then \[\dim(\varphi_{\lambda}^{\oplus q}(R(n)) = \dim(p(R(n))) = m_{\lambda}(R(n)) d_{\lambda}.\] The goal is now to compute $\dim(\varphi_{\lambda}^{\oplus q}(R(n))$.
    			
As $R(n)$ is generated by $x_1, \dots, x_p$, $\varphi_{\lambda}^{\oplus q}(R(n))$ is generated by $\varphi_{\lambda}^{\oplus q}(x_1), \dots, \varphi_{\lambda}^{\oplus q}(x_p)$. But the action of $\bb{S}_n$ on $\mcal{M}_{d_{\lambda}}(\bb{C})^{\oplus q}$ is given by the isomorphism $\varphi^{\oplus q}$. Moreover, the projection $\varphi_{\lambda}^{\oplus q}$ of $\varphi^{\oplus q}$ on $\mcal{M}_{d_{\lambda}}(\bb{C})^{\oplus q}$ is surjective, and \[\varphi_{\lambda}^{\oplus q}(R(n)) = \mcal{M}_{d_{\lambda}}(\bb{C}) \cdot \varphi_{\lambda}^{\oplus q}(x_1) + \dots + \mcal{M}_{d_{\lambda}}(\bb{C}) \cdot \varphi_{\lambda}^{\oplus q}(x_p).\] But because $\varphi_{\lambda}^{\oplus q}(x_i) = (P_{\lambda}(x_i^{(1)}), \dots, P_{\lambda}(x_1)(x_i^{(q)}))$, we have, by Lemma \ref{lemma1Clifton}, \[\dim(\varphi_{\lambda}^{\oplus q}(R(n))) = \rk \left( \left(P_{\lambda}(x^{(j)}_i)\right)_{\substack{1 \leq i \leq p \\ 1 \leq j \leq q}}\right) \cdot d_{\lambda},\] which proves Proposition \ref{propCliftonoperad}.
    			
\end{proof}
			
As an example of application, we refer the reader to \cite{BrDo} where M. Bremner and V. Dotsenko classified the parametrized one-relation operads that are regular using this method. In their case, the matrix encoding all identities is $120 \times 120$ in the polynomial ring $\bb{C}[x_1, x_2, x_3, x_4, x_5, x_6]$, which may seem small enough to compute, but when one wants to compute determinental ideals and their Gröbner bases, this is way too much. The method using representation matrices turns this matrix into five matrices of sizes $5d_{\lambda} \times 5d_{\lambda}$ with $\lambda\vdash 4$, thus $d_{\lambda} = 1, 3, 2, 3, 1$, which they managed to compute.
			
Moreover, this method is not only faster than computing the complete matrix, but it also gives more information on the space of relations.

\subsubsection{The case of $\bb{S}$-bimodules}

Let $n, m \in \bb{N}$. Let us denote $\lambda \vdash m$ and $\mu \vdash n$ by $(\lambda, \mu) \vdash (m, n)$. The regular representation $\bb{C}[\bb{S}_m \times \bb{S}_n^{\op}]$ has the decomposition into isotypic components \[\bb{C}[\bb{S}_m \times \bb{S}_n^{\op}] = \bigoplus_{(\lambda, \mu) \vdash (m, n)} (V_{\lambda} \boxtimes V_{\mu}^{\op})^{\oplus d_{\lambda}d_{\mu}}.\] Indeed we have the sequence of isomorphisms of modules
    			
\begin{align*} 
\bb{C}[\bb{S}_m \times \bb{S}_n^{\op}] &\simeq \bb{C}[\bb{S}_m] \boxtimes \bb{C}[\bb{S}_n^{\op}] \\
&\simeq \left(\bigoplus_{\lambda \vdash m} V_{\lambda}^{\oplus d_{\lambda}}\right) \boxtimes \left(\bigoplus_{\mu \vdash n} (V_{\mu}^{\op})^{\oplus d_{\mu}}\right) \\
&= \bigoplus_{(\lambda, \mu) \vdash (m, n)} V_{\lambda}^{\oplus d_{\lambda}} \boxtimes (V_{\mu}^{\op})^{\oplus d_{\mu}} \\
&= \bigoplus_{(\lambda, \mu) \vdash (m, n)} (V_{\lambda} \boxtimes V_{\mu}^{\op})^{\oplus d_{\lambda}d_{\mu}}.
\end{align*}
    			
\begin{rem} The representations $V_{\lambda} \boxtimes V_{\mu} ^{\op}$ are irreducible and the isomorphism classes of all irreducible representations of $\bb{S}_m \times \bb{S}_n^{\op}$ if $\lambda$ runs through the partitions of $m$ and $\mu$ runs through the partitions of $n$. \end{rem}
    			
Thus let us denote, for a representation $V$ of $\bb{S}_m \times \bb{S}_n^{\op}$, by $m_{\lambda, \mu}(V)$ the multiplicity of $V_{\lambda} \boxtimes V_{\mu}^{\op}$ in $V$ and by $[V]_ {\lambda, \mu}:= (V_{\lambda} \boxtimes V_{\mu})^{\oplus m_{\lambda, \mu}(V)}$ the corresponding isotypic component. In this case we have an isomorphism of algebras \[\varphi : \bb{C}[\bb{S}_m \times \bb{S}_n^{\op}] \rightarrow \bigoplus_{(\lambda, \mu) \vdash (m, n)} \mcal{M}_{d_{\lambda}}(\bb{C}) \otimes \mcal{M}_{d_{\mu}}(\bb{C})\] given by the composition 
\begin{align*}\bb{C}[\bb{S}_m \times \bb{S}_n^{\op}] \overset{\sim}{\rightarrow} \bb{C}[\bb{S}_m] \otimes \bb{C}[\bb{S}_n] &\overset{\sim}{\rightarrow} \left(\bigoplus_{\lambda \vdash m} \mcal{M}_{d_{\lambda}}(\bb{C})\right) \otimes \left(\bigoplus_{\mu \vdash n} \mcal{M}_{d_{\mu}}(\bb{C})\right) \\
&= \bigoplus_{(\lambda, \mu) \vdash (m, n)} \mcal{M}_{d_{\lambda}}(\bb{C}) \otimes \mcal{M}_{d_{\mu}}(\bb{C}).
\end{align*}
    			
Moreover, for any $n, m \in \bb{N}$, we have an isomorphism \[\mcal{M}_n(\bb{C}) \otimes \mcal{M}_m(\bb{C}) \simeq \mcal{M}_{nm}(\bb{C})\] given by the composition
    			
\begin{align*} 
\mcal{M}_n(\bb{C}) \otimes \mcal{M}_m(\bb{C}) \rightarrow &(\bb{C}^n)^* \otimes \bb{C}^n \otimes (\bb{C}^m)^* \otimes \bb{C}^m \\
= &(\bb{C}^n)^* \otimes (\bb{C}^m)^* \otimes \bb{C}^n \otimes \bb{C}^m \rightarrow (\bb{C}^{nm})^* \otimes \bb{C}^{nm} \rightarrow \mcal{M}_{nm}(\bb{C})
\end{align*} which is exactly the Kronecker product of matrices. 
    			
\begin{defin} For $A \in \mcal{M}_n(\bb{C})$ and $B \in \mcal{M}_m(\bb{C})$, the Kronecker product $A\odot B$ is the block matrix in $\mcal{M}_{nm}(\bb{C})$ given by \[A \odot B = \left(A_{ij}B\right)_{1 \leq i, j \leq n}.\] 
\end{defin}

\begin{ex} For example, if $B \in \mcal{M}_n(\bb{C})$ and \[A = \begin{pmatrix}1 & 2 \\ 3 & 4 \end{pmatrix},\] we have \[A \odot B = \begin{pmatrix} B & 2B \\ 3B & 4B \end{pmatrix}.\]
\end{ex}
    			
\begin{property} \label{KroneckerProp}Let $n, m \in \bb{N}$, $M, A \in \mcal{M}_n(\bb{C})$ and $N, B \in \mcal{M}_m(\bb{C})$. We have \[(MA) \odot (NB) = (M \odot N)(A \odot B).\]
\end{property}
    			
\begin{proof} We have \begin{align*} (M \odot N)(A \odot B) &= \left(M_{i,j} N\right)_{i,j} \left(A_{i,j}B\right)_{i,j} \\
&= \left( \left(\sum_{k = 1}^n M_{i,k} A_{k,j}\right)NB\right)_{i,j} \\
&= \left(\left(MA\right)_{i,j} NB\right)_{i,j} \\
&= (MA) \odot (NB).
\end{align*}
\end{proof}
    			
We will denote by $\varphi_{\lambda,\mu}$ the projection of $\varphi$ on $\mcal{M}_{d_{\lambda}}(\bb{C}) \otimes \mcal{M}_{d_{\mu}}(\bb{C})$. This restricts to an isomorphism $P_{\lambda, \mu}$ from $(V_{\lambda} \boxtimes V_{\mu}^{\op})^{d_{\lambda}d_{\mu}}$ to $\mcal{M}_{d_{\lambda}}(\bb{C}) \otimes \mcal{M}_{d_{\mu}}(\bb{C})$. We can now state Proposition \ref{propCliftonoperad} but for $\bb{S}$-bimodules.
    			
\begin{prop} \label{propCliftonproperad} Let $R(m, n)$ be the $\bb{S}_m \times \bb{S}_n^{\op}$-module generated by $x_1, \dots, x_p \in \bb{C}[\bb{S}_m \times \bb{S}_n^{\op}]^{\oplus q}$, and for every $1 \leq i \leq p$, $x_i = (x_i^{(1)}, \dots, x^{(q)}_i)$ their decompositions along $\bb{C}[\bb{S}_m \times \bb{S}_n^{\op}]^{\oplus q}$. Then we have, for every $(\lambda, \mu) \vdash (m, n)$, \[m_{\lambda, \mu}(R(m, n)) = \rk \left( \left(\sum_{k = 1}^{N_{i, j}^{\lambda, \mu}} A_{i,j}^{(k)} \odot B_{i,j}^{(k)}\right)_{\substack{1 \leq i \leq p \\ 1 \leq j \leq q}}\right),\] where \[P_{\lambda, \mu}(x^{(j)}_i) = \sum_{k = 1}^{N_{i,j}^{\lambda, \mu}} A_{i,j}^{(k)} \otimes B_{i,j}^{(k)} \in \mcal{M}_{d_{\lambda}}(\bb{C}) \otimes \mcal{M}_{d_{\mu}}(\bb{C}).\] Moreover, we call the matrix $ \left(\sum_{k = 1}^{N_{i, j}^{\lambda, \mu}} A_{i,j}^{(k)} \odot B_{i,j}^{(k)}\right)_{\substack{1 \leq i \leq p \\ 1 \leq j \leq q}}$ the representation matrix of $R(m, n)$ for $(\lambda, \mu)$.
\end{prop}

\begin{rem} Again, we can consider $R(m, n)$ as a space of relations of a properad in given weight and biarity in order to compute multiplicities of a properad defined by generators and relations.
\end{rem}
    			
Let us first prove a lemma.
    			
\begin{lemme} \label{lemmaKronecker} Let $n, m, N_1, \dots, N_p \in \bb{N}$, and for $1 \leq i \leq p$ and $1 \leq k \leq N_i$, matrices $A_i^{(k)} \in \mcal{M}_n(\bb{C})$ and $B_i^{(k)} \in \mcal{M}_m(\bb{C})$. We have, as vector spaces, \begin{align*}&(\mcal{M}_n(\bb{C}) \otimes \mcal{M}_m(\bb{C})) \cdot \left(\sum_{k = 1}^{N_1} A_1^{(k)} \otimes B_1^{(k)}, \dots, \sum_{k = 1}^{N_p} A_p^{(k)} \otimes B_p^{(k)}\right) \\ &\simeq \mcal{M}_{nm}(\bb{C}) \cdot \left(\sum_{k = 1}^{N_1} A_1^{(k)} \odot B_1^{(k)}, \dots, \sum_{k = 1}^{N_p} A_p^{(k)} \odot B_p^{(k)}\right),\end{align*} where 
    			
{\small \begin{align*} 
&(\mcal{M}_n(\bb{C}) \otimes \mcal{M}_m(\bb{C})) \cdot \left(\sum_{k = 1}^{N_1} A_1^{(k)} \otimes B_1^{(k)}, \dots, \sum_{k = 1}^{N_p} A_p^{(k)} \otimes B_p^{(k)}\right) \\
&= \spn\left(\left(\sum_{k = 1}^{N_1} (MA_1^{(k)}) \otimes (NB_1^{(k)}), \dots, \sum_{k = 1}^{N_p} (MA_p^{(k)}) \otimes (NB_p^{(k)})\right), M \in \mcal{M}_n(\bb{C}), N \in \mcal{M}_m(\bb{C})\right).
\end{align*}}
\end{lemme}
    			
\begin{proof}
{\small \begin{align*}
&(\mcal{M}_n(\bb{C}) \otimes \mcal{M}_m(\bb{C})) \cdot \left(\sum_{k = 1}^{N_1} A_1^{(k)} \otimes B_1^{(k)}, \dots, \sum_{k = 1}^{N_p} A_p^{(k)} \otimes B_p^{(k)}\right) \\
&= \spn\left(\left(\sum_{k = 1}^{N_1} (MA_1^{(k)}) \otimes (NB_1^{(k)}), \dots, \sum_{k = 1}^{N_p} (MA_p^{(k)}) \otimes (NB_p^{(k)})\right), M \in \mcal{M}_n(\bb{C}), N \in \mcal{M}_m(\bb{C})\right) \\
&\simeq \spn\left(\left(\sum_{k = 1}^{N_1} (MA_1^{(k)}) \odot (NB_1^{(k)}), \dots, \sum_{k = 1}^{N_p} (MA_p^{(k)}) \odot (NB_p^{(k)})\right), M \in \mcal{M}_n(\bb{C}), N \in \mcal{M}_m(\bb{C})\right) \\
&= \spn\left(\left(\sum_{k = 1}^{N_1} (M \odot N)(A_1^{(k)} \odot B_1^{(k)}), \dots, \sum_{k = 1}^{N_p} (M \odot N)(A_p^{(k)} \odot B_p^{(k)})\right), M \in \mcal{M}_n(\bb{C}), N \in \mcal{M}_m(\bb{C})\right) \\
&= \spn\left(\left(\sum_{k = 1}^{N_1} M(A_1^{(k)} \odot B_1^{(k)}), \dots, \sum_{k = 1}^{N_p} M(A_p^{(k)} \odot B_p^{(k)})\right), M \in \mcal{M}_{nm}(\bb{C})\right) \\
&= \mcal{M}_{nm}(\bb{C}) \cdot \left(\sum_{k = 1}^{N_1} A_1^{(k)} \odot B_1^{(k)}, \dots, \sum_{k = 1}^{N_1} A_p^{(k)} \odot B_p^{(k)}\right).
\end{align*}}
\end{proof}
    			
\begin{rem} Property \ref{KroneckerProp} and Lemma \ref{lemmaKronecker} can also be proved by the fact that the Kronecker product corresponds to the tensor product of endomorphisms. 
\end{rem}
    			
\begin{proof}[Proof of Proposition \ref{propCliftonproperad}] We have the commutative diagram
    			
\[\xymatrix{R(m, n) \ar@{^{(}->}[r] & \bb{C}[\bb{S}_m \times \bb{S}_n^{\op}]^{\oplus q} \ar[r]^{\varphi_{\lambda, \mu}^{\oplus q}} \ar@<-2pt>[d]_p & (\mcal{M}_{d_{\lambda}}(\bb{C}) \otimes \mcal{M}_{d_{\mu}}(\bb{C}))^{\oplus q} \\ & (V_{\lambda} \boxtimes V_{\mu})^{\oplus qd_{\lambda}d_{\mu}} \ar@<-2pt>[u]_i \ar[ur]^{P_{\lambda, \mu}^{\oplus q}} & }\] Thus $\dim(\varphi_{\lambda, \mu}(R(m, n))) = m_{\lambda, \mu}(R(m, n))d_{\lambda}d_{\mu}$.
    			
Now because $R(m, n)$ is generated as an $\bb{S}_m \times \bb{S}_n^{\op}$ module by $x_1, \dots, x_p$ and $\varphi$ is a morphism of algebras, $\varphi_{\lambda, \mu}^{\oplus q}(R(m, n))$ is generated by $\varphi_{\lambda, \mu}^{\oplus q}(x_1), \dots, \varphi_{\lambda, \mu}^{\oplus q}(x_p)$. But the action of $\bb{S}_m \times \bb{S}_n^{\op}$ on $(\mcal{M}_{d_{\lambda}}(\bb{C}) \otimes \mcal{M}_{d_{\mu}}(\bb{C}))^{\oplus q}$ is given by $\varphi^{\oplus q}$ and the diagonal action on $\bb{C}[\bb{S}_m \times \bb{S}_n^{\op}]^{\oplus q}$, and $\varphi_{\lambda, \mu}$ is a surjection, thus we have 
    			
\[\varphi_{\lambda, \mu}^{\oplus q}(R(m, n)) = \sum_{i = 1}^p(\mcal{M}_{d_{\lambda}}(\bb{C}) \otimes \mcal{M}_{d_{\mu}}(\bb{C})) \cdot (A_{i, 1}^{(k)} \otimes B_{i, 1}^{(k)}, \dots, A_{i, q}^{(k)} \otimes B_{i, q}^{(k)}).\]
    			
By Lemma \ref{lemmaKronecker}, we have 
    			
\[\varphi_{\lambda, \mu}^{\oplus q}(R(m, n)) \simeq \sum_{i = 1}^p \mcal{M}_{d_{\lambda}d_{\mu}}(\bb{C}) \cdot (A_{i, 1}^{(k)} \odot B_{i, 1}^{(k)}, \dots, A_{i, q}^{(k)} \odot B_{i, q}^{(k)}).\] And finally, by Lemma \ref{lemma1Clifton}, the dimension of this space is \[\rk\left(\left(\sum_{k = 1}^{N_{i, j}^{\lambda, \mu}} A_{i, j}^{(k)} \odot B_{i, j}^{(k)}\right)_{i, j}\right) d_{\lambda}d_{\mu},\] which proves Proposition \ref{propCliftonproperad}.
\end{proof}

\section{A family of associative and coassociative properads}

Here we take a very natural family of properads to consider : the one given by an associative product $\mu$, a coassociative coproduct $\Delta$ and one parametrized quadratic relation that corresponds to the rewriting of $\Delta\circ\mu$. The goal is to determine conditions of the parameters for the properads of this family to be non-Koszul. The method we develop in this paper also works for other families of the same type. 

\subsection{Presentation of the family}

We will study the following family of properads. Let $E$ be the $\mathbb{S}$-bimodule \[E = \produit{}{}{} \oplus \coproduit{}{}{},\] $a = (a_1, a_2, a_3, a_4)$ be a $4$-tuple in $\bb{C}^4$, and $R_a$ be the ideal generated by 
			
\begin{align*}
& \peignegauche{}{}{}{} - \peignedroite{}{}{}{} \oplus \copeignegauche{}{}{}{} - \copeignedroite{}{}{}{} \\
&\oplus \jesus{}{}{}{} - a_1 \dromadroite{}{}{}{} - a_2 \dromagauche{}{}{}{} - a_3 \poissongauche{}{}{}{} - a_4 \poissondroite{}{}{}{}.
\end{align*}
        
We denote by $\between_a$ the term \[\between_a := \jesus{}{}{}{} - a_1 \dromadroite{}{}{}{} - a_2 \dromagauche{}{}{}{} - a_3 \poissongauche{}{}{}{} - a_4 \poissondroite{}{}{}{}.\] This notation is inspired from J-L. Loday in \cite{Lo}. Then we denote by $\mcal{P}_a$ the properad $\mcal{F}(E)/R_a$.

Let $\mcal{A}$ be the properad (or operad) of associative algebras and $\mcal{C}$ be the properad of coassociative coalgebras. Let us define the morphism of $\mathbb{S}$-bimodules \[\fonction{\varphi_a}{\mcal{A} \boxtimes \mcal{C}}{\mcal{P}_a}\] defined by the composition \[\mcal{A} \boxtimes \mcal{C} \hookrightarrow \mcal{F}(E) \boxtimes \mcal{F}(E) \rightarrow \mcal{F}(E) \twoheadrightarrow \mcal{P}_a.\]
        		
We can state the following conjecture, which links $\varphi_a$ to Koszulness of $\mcal{P}_a$ and confluence of the system it induces (see Definition \ref{DefConfluence}).
        		
\begin{conj} \label{Conjecture} Let $a \in \bb{C}^4$, the following are equivalent :
        			
\begin{enumerate}
\item the properad $\mcal{P}_a$ induces a confluent system,
\item the morphism $\varphi_a$ is a bijection in weight $3$,
\item the properad $\mcal{P}_a$ is Koszul.
\end{enumerate}
        		
\end{conj}
        		
In fact, we already know that for an operad $\mcal{P}$, if it induces a confluent system (see \cite[Chapter 8]{LoVa}), it is Koszul. But an algebra (thus an operad) can be Koszul even though the system it induces is non confluent, for example the algebra \[A = T(x, y, z)/(x^2 - yx, xz, zy)\] induces a confluent system for the order giving the rule $x^2 \rightarrow yx$ but not for the one giving the rule $yx \rightarrow x^2$ (see \cite{DoRo}). Moreover, (ii) $\Rightarrow$ (iii) has been proven by B. Vallette in \cite{Va} (see Theorem \ref{TheoReplacement}) for properads in general. Unfortunately, the method we develop in this paper has not yet decided rather this conjecture is true or not, but we have some results like the following one and similar ones. 

\begin{theo}
For $a = (a_1, 0, a_3, 0)$, the following are equivalent :
\begin{enumerate}
\item the properad $\mcal{P}_a$ is Koszul,
\item the relation $\between_a$ is one of the following, up to isomorphism : \[\jesus{}{}{}{}\text{ or }\jesus{}{}{}{} - \dromadroite{}{}{}{}.\]
\end{enumerate}
\end{theo}

Most of the time, to compute dimensions or multiplicities, we will study $\mcal{P}_a$ in some weight and biarity. But because the $\bb{S}$-bimodule $E$ is stable by horizontal symmetry, so is $\mcal{F}(E)$ . Thus we can wonder if we can save computations. The relations of associativity and coassociativity are also horizontally symmetric to each other, and the horizontal symmetry of the relation $\between_a$ is $\between_{\tilde{a}}$ with $\tilde{a} := (a_2, a_1, a_3, a_4)$. Thus if we look at the reversed properad $\mcal{P}_a^{\op}$ (see \cite[Section 8]{Va}), we get the properad $\mcal{P}_{\tilde{a}}$.
				
This remark will save a lot of computation for the next sections, because, for example, if we look at the dimension, depending on $a$, of $\mcal{P}_a$ in some weight $w$ and biarity $(m, n)$, we have, as $\bb{S}_n \times \bb{S}_m^{\op}$-modules : \[\mcal{P}^{(w)}_a(n, m) \simeq (\mcal{P}_a^{\op})^{(w)}(m, n) \simeq \left(\mcal{P}^{(w)}_{\tilde{a}}(m, n)\right)^{\op}.\]

\subsection{Confluence}
\label{SubsectionConfluence}

Here we state the conditions on $a$ for $\mcal{P}_a$ to induce a confluent system in the sense given below. Thus we can reformulate the conjecture with these conditions. 
			
\begin{defin}[{\cite[Chapter 1]{Mal}}] \label{DefConfluence} A rewriting system is the data $(A, \rightarrow)$ of a set $A$ and a binary relation $\rightarrow$ on $A$ called the rewriting relation. We say, for $a, b \in A$, that $a$ rewrites to $b$ if there exists a sequence in $A$ \[a \rightarrow a_1 \rightarrow \dots \rightarrow a_n \rightarrow b,\] and we write $a \twoheadrightarrow b$. We say that the system $(A, \rightarrow)$ is confluent if for any tuple $(a, b, c)$ in $A$ such that $b \twoheadleftarrow a \twoheadrightarrow c$, there exists $d \in A$ such that $b \twoheadrightarrow d \twoheadleftarrow c$. See \cite{Mal} for more details on rewriting systems, confluence and linear rewriting. 
\end{defin}
			
We take as rewriting rules on the space $\mcal{F}^{(2)}(E)$ the following ones
        
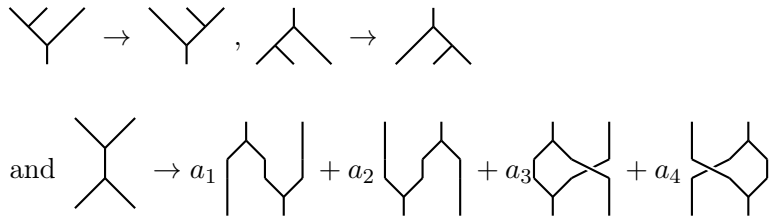
\begin{align*}
&\peignegauche{}{}{}{} \rightarrow \peignedroite{}{}{}{},  \copeignegauche{}{}{}{} \rightarrow \copeignedroite{}{}{}{} \\
&\text{ and }\jesus{}{}{}{} \rightarrow a_1 \dromadroite{}{}{}{} + a_2 \dromagauche{}{}{}{} + a_3 \poissongauche{}{}{}{} + a_4 \poissondroite{}{}{}{}
\end{align*} and take the induced ones on $\mcal{F}^{(3)}(E)$.
        		
\begin{defin} Let $a \in \bb{C}^4$. We say that $\mcal{P}_a$ \textit{induces a confluent system} (in weight $3$) if the rewriting system $\mcal{F}^{(3)}(E)$ with the rewriting relation given above is confluent.
\end{defin}
        		
\begin{rem} This system terminates, thus here confluence is equivalent to convergence (See \cite{Mal}). 
\end{rem}
        
The critical monomials, which are the elements that can be rewritten two different ways, of this rewriting system are \[\longpeignegauche{}{}{}{}{}, \longcopeignegauche{}{}{}{}{}, \deuxtroispremier{}{}{}{}{}\text{ and }\troisdeuxdeuxieme{}{}{}{}{}.\] We already know that starting from \[\longpeignegauche{}{}{}{}{}\text{ and }\longcopeignegauche{}{}{}{}{},\] we have confluence, these are known operadic cases. We only need to check what happens for \[\deuxtroispremier{1}{2}{1}{2}{3}\text{ and }\troisdeuxdeuxieme{1}{2}{1}{2}{3},\] but we will focus on the first one and only give the relations for the other one, because it is obtained from the first one by applying the horizontal symmetry. First we will give some examples.
        		
\begin{exs} For $a = (0, 0, 0, 0)$, we have the rewriting diagram in Figure \ref{a0000},where the top arrow is given by associativity, and the two others are given by the relation $\between_a$. Thus $\mcal{P}_0$ induces a confluent system, but for $a = (2, 0, 0, 0)$, we have the rewriting diagram in Figure \ref{a2000}, where the top path is given by associativity then the relation $\between_a$, and the bottom path is given by the relation $\between_a$ two times followed by associativity. Thus is this case, $\mcal{P}_a$ induces a non confluent system because the two elements at the end of each path cannot be rewritten anymore and are different.
        		
\begin{figure}[!h]
\caption{Rewriting diagram for $a = (0, 0, 0, 0)$}
\label{a0000}
\centering
\[\xymatrix{\deuxtroispremier{1}{2}{1}{2}{3} \ar[rr] \ar[dr] & & \deuxtroisdeuxieme{1}{2}{1}{2}{3} \ar[dl] \\ & 0 & }\]
\end{figure}
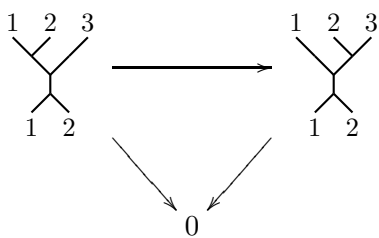  
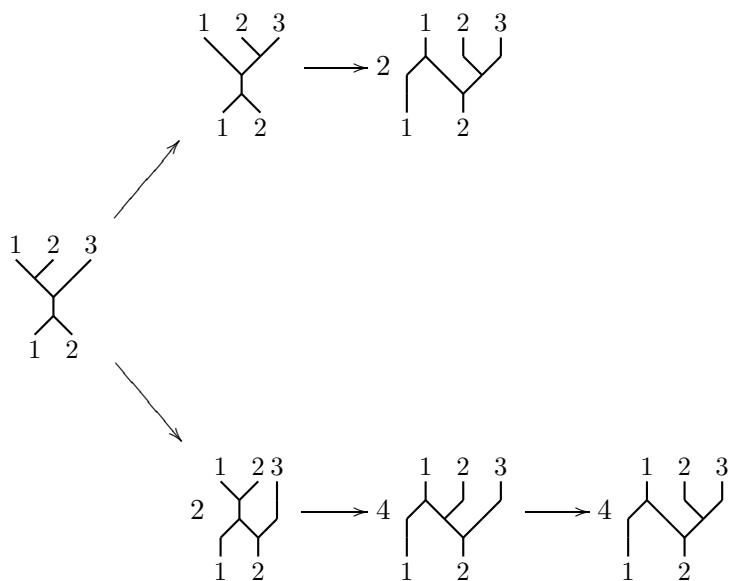
\begin{figure}[!h]
\caption{Rewriting diagram for $a = (2, 0, 0, 0)$} 
\label{a2000}
\centering
\[\xymatrix{ & \deuxtroisdeuxieme{1}{2}{1}{2}{3} \ar[r] & 2\deuxtroisseptieme{1}{2}{1}{2}{3} & \\ \deuxtroispremier{1}{2}{1}{2}{3} \ar[ur] \ar[dr] & & & \\ & 2\deuxtroistroisieme{1}{2}{1}{2}{3} \ar[r] & 4\deuxtroistreizieme{1}{2}{1}{2}{3} \ar[r] & 4\deuxtroisseptieme{1}{2}{1}{2}{3}}\]
\end{figure}
\end{exs}
        		
\begin{prop} Let $a = (a_1, a_2, a_3, a_4) \in \bb{C}^4$ and $\tilde{a} := (a_2, a_1, a_3, a_4)$. The properad $\mcal{P}_a$ induces a confluent system if and only if \[a \in C := \{a \in \bb{C}^4 | \forall i \in \{1, 2, 3, 4\}, a_i = a_i^2 \text{ and } a_2a_4 = a_1a_3 = 0\}\text{ and } \tilde{a} \in C.\] This means that $\mcal{P}_a$ induces a confluent system if and only if for all $i \in \{1, 2, 3, 4\}$, $a_i = a_i^2$ and $a_1a_3 = a_1a_4 = a_2a_3 = a_2a_4 = 0$.
\end{prop}

\begin{rem} Observe that these equations leave the solutions 

\[a = (0, 0, 0, 0), (1, 0, 0, 0), (0, 1, 0, 0), (0, 0, 1, 0), (0, 0, 0, 1), (1, 1, 0, 0) \text{ or }(0, 0, 1, 1)\]
\end{rem}
        		
\begin{proof}
        		
Constructing the rewriting diagram for \[\deuxtroispremier{1}{2}{1}{2}{3}\]\ in the general case, at the end of one path, we get :

\begin{align*}
&a_1\deuxtroisseptieme{1}{2}{1}{2}{3} + a_2\left(a_1\deuxtroisdixneuvieme{1}{2}{1}{2}{3} + a_2\deuxtroisonzieme{1}{2}{1}{2}{3} + a_3\deuxtroisseizieme{1}{2}{1}{3}{2} + a_4\deuxtroisvingtetunieme{1}{2}{1}{3}{2}\right) \\
+& a_3\deuxtroishuitieme{2}{1}{1}{2}{3} + a_4\left(a_1\deuxtroisquinzieme{1}{2}{2}{1}{3} + a_2\deuxtroisvingtetunieme{1}{2}{2}{3}{1} + a_3\deuxtroisvingtdeuxieme{1}{2}{3}{2}{1} + a_4\deuxtroisdouzieme{2}{1}{1}{2}{3}\right),
\end{align*} and at the end of the other path, we get

\begin{align*}
&a_1\left(a_1\deuxtroisseptieme{1}{2}{1}{2}{3} + a_2\deuxtroisdixneuvieme{1}{2}{1}{2}{3} + a_3\deuxtroisvingtieme{1}{2}{2}{1}{3} + a_4\deuxtroisquinzieme{1}{2}{2}{1}{3}\right) + a_2\deuxtroisonzieme{1}{2}{1}{2}{3} \\
+& a_3\left(a_2\deuxtroisvingtieme{1}{2}{3}{1}{2} + a_2\deuxtroisseizieme{1}{2}{1}{3}{2} + a_3\deuxtroishuitieme{2}{1}{1}{2}{3} + a_4\deuxtroisvingtdeuxieme{1}{2}{3}{2}{1}\right) + a_4\deuxtroisdouzieme{2}{1}{1}{2}{3}.
\end{align*}

The difference between these two elements is

\begin{align*}
x_a :=& (a_1 - a_1^2)\deuxtroisseptieme{1}{2}{1}{2}{3} + (a_2^2 - a_2)\deuxtroisonzieme{1}{2}{1}{2}{3} + (a_3 - a_3^2)\deuxtroishuitieme{2}{1}{1}{2}{3} + (a_4^2 - a_4)\deuxtroisdouzieme{2}{1}{1}{2}{3} \\
+& a_2a_4\left(\deuxtroisvingtetunieme{1}{2}{1}{3}{2} + \deuxtroisvingtetunieme{1}{2}{2}{3}{1}\right) - a_1a_3\left(\deuxtroisvingtieme{1}{2}{2}{1}{3} + \deuxtroisvingtieme{1}{2}{3}{1}{2}\right).
\end{align*}
        		
Thus this graph rewrites uniquely if and only if $x_a = 0$ in $\mcal{F}^{(3)}(E)$. In other words, it rewrites uniquely if and only if $a$ is in the set \[C = \{a \in \bb{C}^4 | \forall i \in \{1, 2, 3, 4\}, a_i = a_i^2 \text{ and } a_2a_4 = a_1a_3 = 0\}.\]
        		
Starting from the graph \[\troisdeuxdeuxieme{1}{2}{1}{2}{3},\] the computation is the same but the roles of $a_1$ and $a_2$ are switched, this means that $\mcal{P}_a$ induces a confluent system if and only if $a \in C$ and $\tilde{a} \in C$.  		
\end{proof}

One can now state this new conjecture, which is equivalent to Conjecture \ref{Conjecture}.
        		
\begin{conj} \label{NewConj}%New conjecture.
Let $a$ be a $4$-tuple in $\bb{C}$, then the following are equivalent :
        			
\begin{enumerate}
\item the morphism $\varphi_a$ is an isomorphism in weight $3$,
\item the properad $\mcal{P}_a$ is Koszul.
\end{enumerate}
        			
Moreover, both are true if and only if $\between_a$ is one of the following terms, up to isomorphism : \[\jesus{}{}{}{}, \jesus{}{}{}{} - \dromadroite{}{}{}{}\text{ or }\jesus{}{}{}{} - \dromadroite{}{}{}{} - \dromagauche{}{}{}{}.\]
\end{conj}
        		
Indeed, the properads $\mcal{P}_a$ for $a = (1, 0, 0, 0), (0, 1, 0, 0), (0, 0, 1, 0)$ and $(0, 0, 0, 1)$ are in the same isomorphic class, so are the properads $\mcal{P}_a$ for $a = (1, 1, 0, 0)$ and $(0, 0, 1, 1)$.

\subsection{Koszulness of confluent properads}

\label{SecKoszulness}

Let us recall that the morphism $\varphi_a$ is defined as the composition \[\mcal{A} \boxtimes \mcal{C} \hookrightarrow \mcal{F}(E) \boxtimes \mcal{F}(E) \rightarrow \mcal{F}(E) \twoheadrightarrow \mcal{P}_a.\] The same argument as in the proof of \cite[Proposition 6.2]{EnEt} prove that this morphism is surjective in weight $3$. Thus the question is : is $\varphi_a$ injective in weight $3$ ? We only need to compute the dimensions of $\mcal{A} \boxtimes \mcal{C}$ and $\mcal{P}_a$ in weight $3$, biarity by biarity. The method we will use in this section can be used to compute the biarities $(1, 4)$, $(4, 1)$, $(1, 2)$ and $(2, 1)$, and in these cases, $\varphi_a$ is always injective.
			
The \texttt{Jupyter} notebook \texttt{isobiarity23Method2.ipynb} in \cite{NeWeb} (see Section \ref{SecSagemath}) accompagnies the following proof.

Here we use representation matrices, see Section \ref{SecRepmat}, to compute the representation $R_a^{(3)}(2, 3)$ of the group $\mathbb{S}_2 \times \mathbb{S}_3^{\op}$. $R_a^{(3)}(2, 3)$ is a subspace of $\mcal{F}(E)^{(3)}(2, 3) = \bb{C}[\bb{S}_2 \times \bb{S}_3^{\op}]^{\oplus 22}$, thus we study $R_a^{(3)}(2, 3)$ as a space of relations in $22$ copies of the regular representation of $\mathbb{S}_2 \times \mathbb{S}_3^{\op}$. We have a total of $13$ relations. Thus we get for every $(\lambda, \mu) \vdash (2, 3)$ a $13 d_\lambda d_\mu \times 22 d_\lambda d_\mu$ representation matrix $C_{\lambda, \mu}(R_a^{(3)}(2, 3))$ such that $m_{\lambda, \mu}(R_a^{(3)}(2, 3)) = \rk(C_{\lambda, \mu}(R_a^{(3)}(2, 3)))$. For every $(\lambda, \mu) \vdash (2, 3)$, we compute the partial Smith form (see \cite{BrDo}) of $C_{\lambda, \mu}(R_a^{(3)}(2, 3))$ and get a matrix of the form \[\begin{pmatrix} I_{k_{\lambda, \mu}} & * \\ 0 & B^a_{\lambda, \mu} \end{pmatrix}\] 
			
We simplify the matrices, getting rid of zero and duplicate columns and rows, getting new matrices $SB_{\lambda, \mu}^a$ such that $\rk(SB^a_{\lambda, \mu}) = \rk(B^a_{\lambda, \mu})$. The matrices $SB^a_{\lambda, \mu}$ are given in the Tables \ref{tablebalambdamu21} and \ref{tablebalambdamurest}, and the integers $k_{\lambda, \mu}$ are given in the Table \ref{tableklambdamu}.
			
\begin{table}[h!]
\caption{$SB^a_{\lambda, \mu}$ depending on $\lambda$ for $\mu = (2, 1)$}
\label{tablebalambdamu21}
\centering
\begin{tabular}{r|c}
$\lambda \backslash \mu$ & $(2, 1)$ \\
\hline
$(1, 1)$ & $\begin{pmatrix} -2R_{2, 4} & R_2 & R_2 & R_3 & 0 & -R_{1, 3} & -R_1 & 0 & -R_4 & -R_4 \\ R_{2, 4} & -R_2 & 0 & 0 & R_3 & -R_{1, 3} & 0 & -R_1 & 0 & R_4 \end{pmatrix}$ \\
\hline
$(2)$ & $\begin{pmatrix} 2R_{2, 4} & R_2 & R_2 & R_3 & 0 & R_{1, 3} & R_1 & 0 & R_4 & R_4 \\ R_{2, 4} & -R_2 & 0 & 0 & R_3 & R_{1, 3} & 0 & R_1 & 0 & -R_4 \end{pmatrix}$
\end{tabular}
\end{table}
			
\begin{table}[h!]
\caption{$SB^a_{\lambda, \mu}$ depending on $\lambda$ for $\mu \neq (2, 1)$}
\label{tablebalambdamurest}
\centering
\begin{tabular}{r|c|c}
$\lambda \backslash \mu$ & $(1, 1, 1)$ & $(3)$ \\
\hline
$(1, 1)$ & $\begin{pmatrix} -R_2 & R_3 & -R_1 & -R_4 \end{pmatrix}$ & $\begin{pmatrix} 2R_{2, 4} & -R_2 & R_3 & -2R_{1, 3} & -R_1 & R_4 \end{pmatrix}$ \\
\hline
$(2)$ & $\begin{pmatrix} -R_2 & R_3 & R_1 & R_4 \end{pmatrix}$ & $\begin{pmatrix} -2R_{2, 4} & -R_2 & R_3 & 2R_{1, 3} & R_1 & -R_4 \end{pmatrix}$
\end{tabular}
\end{table}
			
\begin{table}[h!]
\caption{$k_{\lambda, \mu}$ depending on $\lambda$ and $\mu$}
\label{tableklambdamu}
\centering
\begin{tabular}{r|c|c|c}
$\lambda \backslash \mu$ & $(1, 1, 1)$ & $(2, 1)$ & $(3)$ \\
\hline
$(1, 1)$ & $12$ & $24$ & $12$ \\
\hline
$(2)$ & $12$ & $24$ & $12$
\end{tabular}
\end{table}
			
From this we compute $m_{\lambda, \mu}(R_a^{(3)}(2, 3))$ depending on $\mu$ and $a$, see Table \ref{tablebiarity23} where $NC$ stands for Non-Confluent and : 

\begin{enumerate}
\item $NC_{\norm}$ is the set of $4$ tuples $(a_1, a_2, a_3, a_4)$ such that there exists $i \in \{1, 2, 3, 4\}$ such that $a_i \neq a_i^2$.
\item $NC_1$ is the singleton $NC_1 = \{(1, 1, 1, 1)\}$.
\item $NC_{x, 0}$ is the set of $4$-tuples which are not in $C$, $NC_{\norm}$ or $NC_1$, that is the set \[NC_{x, 0} := \{(1, 1, 1, 0), (1, 1, 0, 1), (1, 0, 1, 1), (0, 1, 1, 1), (1, 0, 1, 0), (0, 1, 0, 1)\}.\]
\end{enumerate} 

Moreover, $(\mcal{A} \boxtimes \mcal{C})^{(3)}(2, 3) = \bb{C}[\bb{S}_2 \times \bb{S}_3^{\op}]^{\oplus 10}$, thus the only case where the $\mathbb{S}_2 \times \mathbb{S}_3^{\op}$-module structures of $(\mcal{A} \boxtimes \mcal{C})^{(3)}(2, 3)$ and $\mcal{P}_a^{(3)}(2, 3)$ are the same is when $a \in C$. This proves the following theorem.
			
\begin{table}[h!]
\caption{$m_{\lambda, \mu}(R_a^{(3)}(2, 3))$ depending on $a$ and $\mu$}
\label{tablebiarity23}
\centering
\begin{tabular}{r|c|c|c}
$a\backslash \mu$ & $(3)$ & $(2, 1)$ & $(1, 1, 1)$ \\
\hline
$C$ & $12$ & $24$ & $12$ \\
\hline
$NC_{\norm}$ & $13$ & $26$ & $13$ \\
\hline
$NC_1$ & $12$ & $26$ & $13$ \\
\hline
$NC_{x, 0}$ & $12$ & $25$ & $13$
\end{tabular}
\end{table}

\begin{theo}  
\label{TheoPartial}		
Let $a \in \bb{C}^4$, the following are equivalent :
        			
\begin{enumerate}
\item the properad $\mcal{P}_a$ induces a confluent system,
\item the morphism $\varphi_a$ is an isomorphism of $\bb{S}$-bimodules in weight $3$.
\end{enumerate}
        			
\end{theo}

This shows a part of Conjecture \ref{NewConj}.

\section{Proving non-Koszulness of properads}
\label{SecNonKoszulness}

We can check that all $a \in C$ such that $\tilde{a} \in C$ provide a Koszul properad. In fact if $a = (0, 0, 0, 0)$, we get the properad $\frac{1}{2}\mcal{B}$ (see \cite{MaVo}), if $a = (1, 1, 0, 0)$, we get $\varepsilon \mcal{B}$ (see \cite[Corollary 8.5]{Va}) and if $a = (1, 0, 0, 0)$, Theorem \ref{TheoPartial} and Theorem \ref{TheoReplacement} show that $\mcal{P}_a$ is Koszul. Thus we can state the following.
		
\begin{theo}
Let $a \in \bb{C}^4$. The following are equivalent :
			
\begin{enumerate}
\item the properad $\mcal{P}_a$ induces a confluent system,
\item the morphism $\varphi_a$ is an isomorphism of $\bb{S}$-bimodules in weight $3$.
\end{enumerate}
			
Moreover, in that case, the properad $\mcal{P}_a$ is Koszul.
			
\end{theo}

This is again another step to prove Conjecture \ref{NewConj}. Now the question is, if $a \notin C$ or $\tilde{a} \notin C$, can we show that $\mcal{P}_a$ is not Koszul ? 

\subsection{The Koszul complex}

We have the following sequence of implications :
		
\begin{align*}
\mcal{P}_a\text{ is Koszul } &\implies \mcal{P}_a \boxtimes \mcal{P}_a^{\text{\textexclamdown}}\text{ is acyclic } \implies (\mcal{P}_a \boxtimes \mcal{P}_a^{\text{\textexclamdown}})^{(4)}(2, 4)\text{ is acyclic } \\ 
&\implies\text{ for all }(\lambda, \mu) \vdash (2, 4)\text{, the isotypic component } \\
&[(\mcal{P}_a \boxtimes \mcal{P}_a^{\text{\textexclamdown}})^{(4)}(2, 4)]_{(\lambda, \mu)}\text{ is acyclic } \\
&\implies\text{ for all }(\lambda, \mu) \vdash (2, 4)\text{, the Euler characteristic of } \\
&[(\mcal{P}_a \boxtimes \mcal{P}_a^{\text{\textexclamdown}})^{(4)}(2, 4)]_{(\lambda, \mu)}\text{ is zero.}
\end{align*}
		
Thus we will consider this last necessary condition and call it \textit{criterion}. The goal of this section is to compute the multiplicities of every chain of the Koszul chain complex of $\mcal{P}_a$ in biarity $(2, 4)$ and weight $4$. We can look for multiplicities instead of dimensions because every map in the complex is equivariant. We look in weight $4$ because in lower weights, this criterion is always true. We also look for biarity $(2, 4)$ because it seems to be the simplest one to study among the ones where we can find something interesting. We could look for biarity $(3, 3)$, or higher weights, the method would be the same but the computations would be too long. 
		
For the \texttt{SageMath} computations, the file \texttt{Koszulcpxbiarity24.ipynb} in \cite{NeWeb} accompanies this section, see Section \ref{SecSagemath}.
			
Let us study the Koszul complex of $\mcal{P}_a$ in weight $4$ biarity $(2, 4)$ :
			
{\small \[0 \rightarrow \overbrace{(\mcal{P}_a^{\text{\textexclamdown}})^{(4)}(2, 4)}^{\text{Degree } 4} \rightarrow \overbrace{\underbrace{\mcal{P}_a}_{1} \boxtimes \underbrace{\mcal{P}_a^{\text{\textexclamdown}}}_{3}(2, 4)}^{\text{Degree } 3} \rightarrow \overbrace{\underbrace{\mcal{P}_a}_{2} \boxtimes \underbrace{\mcal{P}_a^{\text{\textexclamdown}}}_{2}(2, 4)}^{\text{Degree } 2} \rightarrow \overbrace{\underbrace{\mcal{P}_a}_{3} \boxtimes \underbrace{\mcal{P}_a^{\text{\textexclamdown}}}_{1}(2, 4)}^{\text{Degree } 1} \rightarrow \overbrace{\mcal{P}_a^{(4)}(2, 4)}^{\text{Degree } 0} \rightarrow 0.\]}
			
Let us first introduce some notations.
			
\begin{notas}
Let $(\lambda, \mu) \vdash (2, 4)$. We set
				
\begin{enumerate}
\item $M_{\lambda, \mu} := m_{\lambda,\mu}(\bb{C}[\bb{S}_2 \times \bb{S}_4^{\op}])$.
\item $X^a := (\mcal{P}_a^{(3)}(2, 3) \boxtimes_{(1)} E(1, 2))(2, 4)$ and $x^a_{\lambda, \mu} := m_{\lambda,\mu}(X^a)$.
\item $Y^a := (E(1 ,2) \boxtimes_{(1)} (\mcal{P}_a^{\text{\textexclamdown}})^{(3)}(2, 3))(2, 4)$ and $y^a_{\lambda, \mu} := m_{\lambda,\mu}(Y^a)$.
\item $(M\mcal{P}_a)_{\lambda, \mu} := m_{\lambda,\mu}(\mcal{P}_a^{(4)}(2, 4))$ and $(M\mcal{P}_a^{\text{\textexclamdown}})_{\lambda, \mu} := m_{\lambda,\mu}((\mcal{P}_a^{\text{\textexclamdown}})^{(4)}(2, 4))$.
\item $(MR_a)_{\lambda, \mu} := m_{\lambda, \mu}(R_a^{(4)}(2, 4))$ and $(MR_a^{\perp})_{\lambda, \mu} := m_{\lambda, \mu}((R_a^{\perp})^{(4)}(2, 4))$
\end{enumerate}
\end{notas}
			
As the criterion we are looking for only involves multiplicities of the chains, we will study this complex degree by degree.

\subsection{Internal chains}

\subsubsection{Degree 1}
			
Looking at the combinatorics with the weight and the biarity, we can see that 
				
\begin{align*}
\underbrace{\mcal{P}_a}_{3} \boxtimes \underbrace{\mcal{P}_a^{\text{\textexclamdown}}}_{1}(2, 4) =&  X^a \oplus (\mcal{P}_a^{(3)}(1, 4) \boxtimes_{(1)} E(2, 1))(2, 4) \\
& \oplus ((E(1, 2), \mcal{P}_a^{(2)}(1, 3)) \boxtimes_{(1)} E(2, 1))(2, 4).
\end{align*} According to Notation \ref{NotaInfinitesimal}, these subspaces correspond respectively to the diagrams 
				
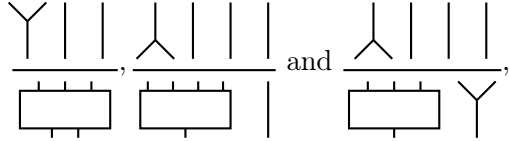
\[\PtBEu, \PtBEd\text{ and } \PdEBE,\] the first space has multiplicity $x^a_{\lambda, \mu}$, which we will compute in \ref{SecMultInter}, the second one is given by Corollary \ref{CoroComposBlocks}, and the last one can be computed with the same methods as in Theorem \ref{TheoComposBlocks}, it is equal to
				
{\scriptsize
\begin{align*}
& \left(\bigoplus_{\substack{\overline{k} =(1, \dots, 2, \dots, 1) \\ \nor{\overline{k}} = 5}} \bb{C}[\bb{S}_2] \otimes (E(1, 2) \otimes \mcal{P}_a^{(2)}(1, 3)) \otimes_{\bb{S}_2 \times \bb{S}_3} \bb{C}[\bb{S}^c_{(2, 3), \overline{k}}] \otimes_{\bb{S}_{\overline{k}}} (\bb{C} \otimes \dots \otimes E(2, 1) \otimes \dots \otimes \bb{C}) \otimes \bb{C}[\bb{S}_4]\right. \\
& \left. \oplus \bigoplus_{\substack{\overline{k} = (1, \dots, 2, \dots, 1) \\ \nor{\overline{k}} = 5}} \bb{C}[\bb{S}_2] \otimes (\mcal{P}_a^{(2)}(1, 3) \otimes E(1, 2)) \otimes_{\bb{S}_3 \times \bb{S}_2} \bb{C}[\bb{S}^c_{(3, 2), \overline{k}}] \otimes_{\bb{S}_{\overline{k}}} (\bb{C} \otimes \dots \otimes E(2, 1) \otimes \dots \otimes \bb{C}) \otimes \bb{C}[\bb{S}_4]\right)/\sim.
\end{align*}} But again we can reorder the blocks and use the fact that the $\bb{S}$-bimodules involved in these arities are regular representations to get
				
\begin{align*}
&\left(\bb{C}[\bb{S}_2] \otimes (\bb{C}[\bb{S}_2] \otimes \bb{C}[\bb{S}_3]) \otimes_{\bb{S}_2 \times \bb{S}_3} \bb{C}[\bb{S}^c_{(2, 3), (2, 1, 1, 1)}] \otimes_{\bb{S}_2} \bb{C}[\bb{S}_2] \otimes \bb{C}[\bb{S}_4]\right)/\sim' \\
=& \left(\bb{C}[\bb{S}_2] \otimes \bb{C}[\bb{S}^c_{(2, 3), (2, 1, 1, 1)}] \otimes \bb{C}[\bb{S}_4]\right)/\sim'.
\end{align*} Here we have \[\bb{S}^c_{(2, 3), (2, 1, 1, 1)} \simeq \{1, 2, 3\} \times \{1, 2\} \times \{1, 2\} \times \bb{S}_3,\] with a right action of $\bb{S}_3$ corresponding to the relation $\sim'$, which gives 
				
\begin{align*}((E(1, 2), \mcal{P}_a^{(2)}(1, 3)) \boxtimes_{(1)} E(2, 1))(2, 4) &= \bb{C}[\bb{S}_2] \otimes \bb{C}[\bb{S}^c_{(2, 3), (2, 1, 1, 1)}] \otimes_{\bb{S}_3} \bb{C}[\bb{S}_4] \\
&= \bb{C}[\bb{S}_2] \otimes \bb{C}^{12} \otimes \bb{C}[\bb{S}_4] \\
&= \bb{C}[\bb{S}_2 \times \bb{S}_4^{\op}] ^{\oplus 12}.
\end{align*}
				
Finally, we have
				
\[m_{\lambda, \mu}((\underbrace{\mcal{P}_a}_{3} \boxtimes \underbrace{\mcal{P}_a^{\text{\textexclamdown}}}_{1})(2, 4)) = x^a_{\lambda, \mu} + 20M_{\lambda, \mu}.\] 
        			
\subsubsection{Degree 2}
        		
Similarly to degree $1$, we can see that the space $\underbrace{\mcal{P}_a}_{2} \boxtimes \underbrace{\mcal{P}_a^{\text{\textexclamdown}}}_{2}(2, 4)$ is
        			
\begin{align} \label{deg2}
& \left(E(1, 2) \boxtimes_{(2), (1, 1)} (E(1, 2), E(2, 1))\right)(2, 4) \oplus \left(E(1, 2) \boxtimes_{(2), (1)} (\mcal{P}_a^{\text{\textexclamdown}})^{(2)}(2, 2)\right)(2, 4) \\
& \oplus \left(\mcal{P}_a^{(2)}(2, 2) \boxtimes_{(1), (2)} E(1, 2)\right)(2, 4) \oplus \left(\mcal{P}_a^{(2)}(1, 3) \boxtimes_{(1)} (E(1, 2), E(2, 1))\right)(2, 4) \\
& \oplus \left(\mcal{P}_a^{(2)}(2, 2) \boxtimes_{(1)} (\mcal{P}_a^{\text{\textexclamdown}})^{(2)}(1, 3)\right)(2, 4) \oplus \left(\mcal{P}_a^{(2)}(1, 3) \boxtimes_{(1)} (\mcal{P}_a^{\text{\textexclamdown}})^{(2)}(2, 2)\right)(2, 4).
\end{align} According to Notation \ref{NotaInfinitesimal}, these subspaces correspond respectively to the diagrams 
        			
\[\EEBEE, \PdBEE, \EEBPdu, \EEBPdd, \PdBPdu\text{ and }\PdBPdd.\] Thus, by Corollary \ref{CoroComposBlocks} and the previous method, we have 
				
\[m_{\lambda, \mu}((\underbrace{\mcal{P}_a}_{3} \boxtimes \underbrace{\mcal{P}_a^{\text{\textexclamdown}}}_{1})(2, 4)) = 42M_{\lambda, \mu}.\]
				
Let us give some details on the computation of the first space in \ref{deg2}. This space is equal to
				
\begin{align*}
&\left(\bb{C}[\bb{S}_2] \otimes (E(1, 2) \otimes E(1, 2) \otimes_{\bb{S}_2 \times \bb{S}_2} \bb{C}[\bb{S}^c_{(2, 2), (1, 1, 2)}] \otimes_{\bb{S}_2} (E(1, 2) \otimes E(2, 1)) \otimes_{\bb{S}_2} \bb{C}[\bb{S}_4]\right)/\sim' \\
=& \left(\bb{C}[\bb{S}_2] \otimes \bb{C}[\bb{S}^c_{(2, 2), (1, 1, 2)}] \otimes \bb{C}[\bb{S}_4]\right)/\sim',
\end{align*} where the relation $\sim'$ corresponds to the exchange between the two products of $E(1, 2)$. We have \[\bb{S}^c_{(2, 2), (1, 1, 2)} \simeq \{1, 2\} \times \{1, 2\} \times \{1, 2\} \times \bb{S}_2,\] where an element $(a, b_1, b_2, \tau) \in \bb{S}^c_{(2, 2), (1, 1, 2)}$ corresponds to the choice $a$ of an output of the coproduct, an input $b_1$ of the first product which is linked to the output $a$, an input $b_2$ of the second product which is linked to the other output, and a permutation $\tau \in \bb{S}_2$ for the rest of the permutation. For example, $(2, 1, 2, (1, 2))$ corresponds to the connected permutation $[3, 2, 4, 1]$. This set is endowed with a left action of $\bb{S}_2$ given by \[\sigma \cdot (a, b_1, b_2, \tau) := (\sigma \cdot a, b_{\sigma(1)}, b_{\sigma(2)}, \tau),\] and a right action given by multiplication with $\tau$. Thus, we have 
				
\begin{align*}
\left(E(1, 2) \boxtimes_{(2), (1, 1)} (E(1, 2), E(2, 1))\right)(2, 4) &= \bb{C}[\bb{S}_2] \otimes_{\bb{S}_2} \bb{C}[\bb{S}^c_{(2, 2), (1, 1, 2)}] \otimes \bb{C}[\bb{S}_4] \\
&= \bb{C}[\bb{S}_2] \otimes \bb{C}^8 \otimes \bb{C}[\bb{S}_4] \\
&= \bb{C}[\bb{S}_2 \times \bb{S}_4^{\op}]^{\oplus 8}.
\end{align*}
        			
\subsubsection{Degree 3}
        		
We have
        		
\begin{align*}
\underbrace{\mcal{P}_a}_{2} \boxtimes \underbrace{\mcal{P}_a^{\text{\textexclamdown}}}_{2}(2, 4) =& Y^a \oplus \left(E(2, 1) \boxtimes_{(1)} (\mcal{P}_a^{\text{\textexclamdown}})^{(3)}(1, 4)\right)(2, 4) \\
& \oplus \left(E(1, 2) \boxtimes_{(1)} ((\mcal{P}_a^{\text{\textexclamdown}})^{(2)}(2, 2), E(1, 2))\right)(2, 4) \\
& \oplus \left(E(1, 2) \boxtimes_{(1)} ((\mcal{P}_a^{\text{\textexclamdown}})^{(2)}(1, 3), E(2, 1))\right)(2, 4).
\end{align*} According to Notation \ref{NotaInfinitesimal}, these subspaces correspond respectively to the diagrams 
        		
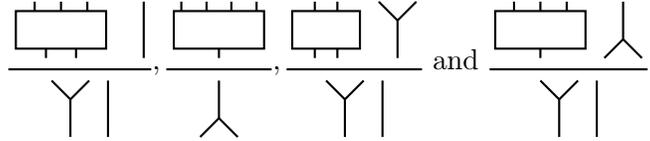
\[\EBPtu, \EBPtd, \EBPdEu\text{ and }\EBPdEd.\] Thus, by Corollary \ref{CoroComposBlocks} and the previous method, we have 
				
\[m_{\lambda, \mu}((\underbrace{\mcal{P}_a}_{1} \boxtimes \underbrace{\mcal{P}_a^{\text{\textexclamdown}}}_{3})(2, 4)) = y^a_{\lambda, \mu} + 9M_{\lambda, \mu}.\]
				
\subsubsection{Calculation of the multiplicities}
\label{SecMultInter}

First, we know the multiplicities $M_{\lambda, \mu}$ of the regular representation, given in Table \ref{MultReg}. Then we compute $x_{\lambda, \mu}^a$ and $y_{\lambda, \mu}^a$ using Pieri's rules.
        		
\begin{table}[h!]
\caption{Values of $M_{\lambda, \mu}$ depending on $\mu$}
\label{MultReg}
\centering
\begin{tabular}{r|c|c|c|c|c}
$\mu$ & $(4)$ & $(3, 1)$ & $(2, 2)$ & $(2, 1, 1)$ & $(1, 1, 1, 1)$ \\
\hline
$M_{\lambda, \mu}$ & $1$ & $3$ & $2$ & $3$ & $1$
\end{tabular}
\end{table}

\textbf{Computation of $x_{\lambda, \mu}^a$ :} We denote, for $(\alpha, \beta) \vdash (2, 3)$, $A_{\alpha, \beta}^a := m_{\alpha, \beta}(\mcal{P}_a^{(3)}(2, 3))$. Thus we have \[\mcal{P}_a^{(3)}(2, 3) = \bigoplus_{(\alpha, \beta) \vdash (2, 3)} V_{\alpha, \beta}^{\oplus A_{\alpha, \beta}^a} = \bigoplus_{(\alpha, \beta) \vdash (2, 3)} (V_{\alpha} \boxtimes V_{\beta})^{\oplus A_{\alpha, \beta}^a}.\] According to Theorem \ref{TheoComposBlocks}, we have \[X^a = \bigoplus_{(\alpha, \beta) \vdash (2, 3)} \left(V_{\alpha} \boxtimes (V_{\beta}\downarrow^{\bb{S}_3}_{\bb{S}_2} \sqcup_r \bb{C}[\bb{S}_2])\right)^{\oplus A_{\alpha, \beta}}.\] Thus by Pieri's rules :
        		
\begin{align*}
X^a &= \bigoplus_{(\alpha, \beta) \vdash (2, 3)} \bigoplus_{\substack{\gamma \vdash 2 \\ \gamma \leq \beta}} \left(V_{\alpha} \boxtimes (V_{\gamma} \sqcup_r \bb{C}[\bb{S}_2])\right)^{\oplus A_{\alpha, \beta}} \\
&= \bigoplus_{(\alpha, \beta) \vdash (2, 3)} \bigoplus_{\substack{\gamma \vdash 2 \\ \gamma \leq \beta}} \left(\left(V_{\alpha} \boxtimes (V_{\gamma} \sqcup_r V_{(1, 1)})\right) \oplus \left(V_{\alpha} \boxtimes (V_{\gamma} \sqcup_r V_{(2)})\right)\right)^{\oplus A_{\alpha, \beta}} \\
&= \bigoplus_{(\alpha, \beta) \vdash (2, 3)} \bigoplus_{\substack{\gamma \vdash 2 \\ \gamma \leq \beta}} \left(\left(\bigoplus_{\substack{\delta \vdash 4 \\ \delta \leq^l \gamma}} V_{\alpha} \boxtimes V_{\delta}\right) \oplus \left(\bigoplus_{\substack{\delta \vdash 4 \\ \delta \leq^c \gamma}} V_{\alpha} \boxtimes V_{\delta}\right)\right)^{\oplus A_{\alpha, \beta}}.
\end{align*}
        		
In terms of Young diagrams, the partition $\delta$ is obtained from $\beta$ by removing one box and adding two boxes, but not both on the same column for the first direct sum and not both on the same row for the second one. For every $\beta \vdash 3$ and $\delta \vdash 4$, we count the number of ways to get $\delta$ from $\beta$ these ways in Table \ref{deltabeta}. We also remind the values of $A^a_{\alpha, \beta}$ from Section \ref{SecKoszulness} for every $a \in \bb{C}^4$ and $(\alpha,\beta) \vdash (2, 3)$, depending on $a$ and $\beta$, see Table \ref{Aa}. Finally, we obtain the values of $x^a_{\lambda, \mu}$ depending on $a$ and $\mu$, see Table \ref{xalambdamu}.
        		
\begin{table}[h!]
\caption{Number of ways to get $\delta$ from $\beta$ by removing one box and adding two boxes, not on the same column (left), not on the same row (right)}
\label{deltabeta}
\centering
\begin{tabular}{r|c|c|c|c|c}
$\beta \backslash \delta$ & \ydiagram{4} & \ydiagram{3, 1} & \ydiagram{2, 2} & \ydiagram{2, 1, 1} & \ydiagram{1, 1, 1, 1} \\
\hline
\ydiagram{3} & $\begin{matrix} 1 & 0 \end{matrix}$ & $\begin{matrix} 1 & 1 \end{matrix}$ & $\begin{matrix} 1 & 0 \end{matrix}$ & $\begin{matrix} 0 & 1 \end{matrix}$ & $\begin{matrix} 0 & 0 \end{matrix}$ \\
\hline
\ydiagram{2, 1} & $\begin{matrix} 1 & 0 \end{matrix}$ & $\begin{matrix} 2 & 1 \end{matrix}$ & $\begin{matrix} 1 & 1 \end{matrix}$ & $\begin{matrix} 1 & 2 \end{matrix}$ & $\begin{matrix} 0 & 1 \end{matrix}$ \\
\hline
\ydiagram{1, 1, 1} & $\begin{matrix} 0 & 0 \end{matrix}$ & $\begin{matrix} 1 & 0 \end{matrix}$ & $\begin{matrix} 0 & 1 \end{matrix}$ & $\begin{matrix} 1 & 1 \end{matrix}$ & $\begin{matrix} 0 & 1 \end{matrix}$
\end{tabular}
\end{table}
        		
\begin{table}[h!]
\caption{Values of $A^a_{\alpha, \beta}$ depending on $a$ and $\beta$}
\label{Aa}
\centering
\begin{tabular}{r|c|c|c|c|c}
$a\in \backslash \beta$ & \ydiagram{3} & \ydiagram{2, 1} & \ydiagram{1, 1, 1} \\
\hline
$C$ & $10$ & $20$ & $10$ \\
\hline
$NC_{\norm}$ & $9$ & $18$ & $9$ \\
\hline
$NC_1$ & $10$ & $18$ & $9$ \\
\hline
$NC_{x, 0}$ & $10$ & $19$ & $9$
\end{tabular}
\end{table}
        		
\begin{table}[h!]
\caption{$x^a_{\lambda,\mu}$ depending on $a$ and $\mu$}
\label{xalambdamu}
\centering
\begin{tabular}{r|c|c|c|c|c}
$a\in \backslash \mu$ & \ydiagram{4} & \ydiagram{3, 1} & \ydiagram{2, 2} & \ydiagram{2, 1, 1} & \ydiagram{1, 1, 1, 1} \\
\hline
$C$ & $30$ & $90$ & $60$ & $90$ & $30$ \\
\hline
$NC_{\norm}$ & $27$ & $81$ & $54$ & $81$ & $27$ \\
\hline
$NC_1$ & $28$ & $83$ & $55$ & $82$ & $27$ \\
\hline
$NC_{x, 0}$ & $29$ & $86$ & $57$ & $85$ & $28$
\end{tabular}
\end{table}
       
\textbf{Computation of $y_{\lambda, \mu}^a$ :} We now denote, for $(\alpha, \beta) \vdash (2, 3)$, $B_{\alpha,\beta}^a := m_{\alpha, \beta}(\mcal{P}_a^{\text{\textexclamdown}(3)}(2, 3))$. We have \[\mcal{P}_a^{\text{\textexclamdown}(3)}(2, 3) = \bigoplus_{(\alpha, \beta) \vdash (2, 3)} V_{\alpha, \beta}^{\oplus B_{\alpha, \beta}^a} = \bigoplus_{(\alpha, \beta) \vdash (2, 3)} (V_{\alpha} \otimes V_{\beta})^{\oplus B_{\alpha, \beta}^a}.\] We have, by Theorem \ref{TheoComposBlocks}, \[Y^a = \bigoplus_{(\alpha, \beta) \vdash (2, 3)} \left((\bb{C} \sqcup_l \ ^{\bb{S}_2}_{\bb{S}_1}\downarrow V_{\alpha}) \boxtimes (\bb{C}[\bb{S}_2] \downarrow^{\bb{S}_2}_{\bb{S}_1} \sqcup_r V_{\beta}\right)^{\oplus B_{\alpha, \beta}}.\] Thus, by Pieri's rules,
        		
\begin{align*}
Y^a &= \bigoplus_{(\alpha, \beta) \vdash (2, 3)} \left((\bb{C} \sqcup_l \bb{C}) \boxtimes (\bb{C} \sqcup_r V_{\beta})\right)^{\oplus 2B_{\alpha, \beta}} \\
&= \bigoplus_{\beta \vdash 3} \left(\bb{C}[\bb{S}_2] \boxtimes (\bb{C} \sqcup_r V_{\beta})\right)^{\oplus 4B_{\alpha, \beta}} \\
&= \bigoplus_{\beta \vdash 3} \bigoplus_{\substack{\gamma \vdash 4 \\ \gamma \geq \beta}} \left(\bb{C}[\bb{S}_2] \boxtimes V_{\gamma} \right)^{\oplus 4B_{\alpha, \beta}}.
\end{align*}
			
Here, in terms of Young diagrams, $\gamma$ is obtained from $\beta$ by adding one box. For every $\gamma \vdash 4$ and $\beta \vdash 3$, we count the number of ways to get $\gamma$ from $\beta$ this way in Table \ref{gammabeta}. We also have, by \texttt{SageMath} computations and Proposition \ref{PropDual}, the values of $B^a_{\alpha, \beta}$ depending on $\beta$ and $a$, see Table \ref{Ba}. Then we get the values of $y^a_{\lambda, \mu}$ depending on $a$ and $\mu$, see Table \ref{yalambdamu}, and the values of $x^a_{\lambda, \mu} + y^a_{\lambda, \mu}$ depending on $a$ and $\mu$, see Table \ref{xplusy}.
			
\begin{table}[h!]
\caption{Number of ways to get $\gamma$ from $\beta$ adding one box}
\label{gammabeta}
\centering
\begin{tabular}{r|c|c|c|c|c}
$\beta \backslash \gamma$ & \ydiagram{4} & \ydiagram{3, 1} & \ydiagram{2, 2} & \ydiagram{2, 1, 1} & \ydiagram{1, 1, 1, 1} \\
\hline
\ydiagram{3} & $1$ & $1$ & $0$ & $0$ & $0$ \\
\hline
\ydiagram{2, 1} & $0$ & $1$ & $1$ & $1$ & $0$ \\
\hline
\ydiagram{1, 1, 1} & $0$ & $0$ & $0$ & $1$ & $1$
\end{tabular}
\end{table}
			
\begin{table}[h!]
\caption{Values of $B^a_{\alpha, \beta}$ depending on $a$ and $\beta$}
\label{Ba}
\centering
\begin{tabular}{r|c|c|c|c|c}
$a\in \backslash \beta$ & \ydiagram{3} & \ydiagram{2, 1} & \ydiagram{1, 1, 1} \\
\hline
$C$ & $1$ & $2$ & $1$ \\
\hline
$NC_{\norm}$ & $0$ & $0$ & $0$ \\
\hline
$NC_1$ & $0$ & $0$ & $1$ \\
\hline
$NC_{x, 0}$ & $0$ & $1$ & $1$
\end{tabular}
\end{table}
        		
\begin{table}[h!]
\caption{$y^a_{\lambda,\mu}$ depending on $a$ and $\mu$}
\label{yalambdamu}
\centering
\begin{tabular}{r|c|c|c|c|c}
$a\in \backslash \mu$ & \ydiagram{4} & \ydiagram{3, 1} & \ydiagram{2, 2} & \ydiagram{2, 1, 1} & \ydiagram{1, 1, 1, 1} \\
\hline
$C$ & $4$ & $12$ & $8$ & $12$ & $4$ \\
\hline
$NC_{\norm}$ & $0$ & $0$ & $0$ & $0$ & $0$ \\
\hline
$NC_1$ & $0$ & $0$ & $0$ & $4$ & $4$ \\
\hline
$NC_{x, 0}$ & $0$ & $4$ & $4$ & $8$ & $4$
\end{tabular}
\end{table}
			
\begin{table}[h!]
\caption{$x^a_{\lambda,\mu} + y^a_{\lambda,\mu}$ depending on $a$ and $\mu$}
\label{xplusy}
\centering
\begin{tabular}{r|c|c|c|c|c}
$a\in \backslash \mu$ & \ydiagram{4} & \ydiagram{3, 1} & \ydiagram{2, 2} & \ydiagram{2, 1, 1} & \ydiagram{1, 1, 1, 1} \\
\hline
$C$ & $34$ & $102$ & $68$ & $102$ & $34$ \\
\hline
$NC_{\norm}$ & $27$ & $81$ & $54$ & $81$ & $27$ \\
\hline
$NC_1$ & $28$ & $83$ & $55$ & $86$ & $31$ \\
\hline
$NC_{x, 0}$ & $29$ & $90$ & $61$ & $93$ & $32$
\end{tabular}
\end{table}

\subsection{External chains}

\subsubsection{Degree 0}
				
Here the method is exactly the same as in Section \ref{SecKoszulness}, we have \[(M\mcal{P}_a)_{\lambda, \mu} = 93M_{\lambda, \mu} - (MR_a)_{\lambda, \mu}.\] We compute the partial Smith form of the representation matrices for $R_a^{(4)}(2, 4)$ and get for every $(\lambda, \mu) \vdash (2, 4)$ a matrix of the form \[C_{\lambda, \mu}(R_a^{(4)}(2, 4)) := \begin{pmatrix} I_{k_{\lambda, \mu}} & * \\ 0 & B^a_{\lambda, \mu} \end{pmatrix}\] such that $(MR_a)_{\lambda, \mu} = k_{\lambda, \mu} + r^a_{\lambda, \mu}$ where $r^a_{\lambda, \mu} := \rk(B^a_{\lambda, \mu})$. Here the issue is that the representation matrices are too big to be studied as simply as in \ref{SecKoszulness}, thus we have to find new strategies. In some cases, the representation matrices are small enough to compute Gröbner bases, as we will see in Section \ref{subsubsecGrob}.
				 			
\subsubsection{Degree 4}
        			
Using Proposition \ref{PropDual}, we do the same here as for degree $0$, but for the properad $\mcal{P}_a^{\text{!}} = \mcal{F}(\Sigma E)/(\Sigma^2 R^{\perp})$, with 
        			
\begin{align*}
R^{\perp} =& \left(\peignegauche{}{}{} - \peignedroite{}{}{}\right) \oplus \left(\copeignegauche{}{}{} - \copeignedroite{}{}{}\right) \oplus \left(a_1\jesus{}{}{}{} - \dromadroite{}{}{}{}\right) \\
& \oplus \left(a_2\jesus{}{}{}{} - \dromagauche{}{}{}{}\right)\oplus \left(a_3\jesus{}{}{}{} - \poissongauche{}{}{}{}\right)\oplus \left(a_4\jesus{}{}{}{} - \poissondroite{}{}{}{}\right).
\end{align*}
        			
Thus we find 
        			
\begin{align*}
(M\mcal{P}_a^{\text{\textexclamdown}})_{\lambda, \mu} &= 93M_{\lambda, \mu} - m_{\lambda', \mu'}((R^{\perp})^{(4)}(2, 4)) \\
&= 93M_{\lambda, \mu} - k^{\perp}_{\lambda', \mu'} - r^{a, \perp}_{\lambda', \mu'},
\end{align*} where the partial Smith form of the representation matrices for $(R^{\perp})^{(4)}(2, 4)$ and $(\lambda, \mu) \vdash (2, 4)$ are of the form 
        			
\[\begin{pmatrix} I_{k^{\perp}_{\lambda, \mu}} & * \\ 0 & B^{a, \perp}_{\lambda, \mu} \end{pmatrix},\] and $r^{a, \perp}_{\lambda, \mu} := \rk(B^{a, \perp}_{\lambda, \mu})$.

\subsubsection{Computing the multiplicities}
\label{subsubsecGrob}

Now we can say that if $\mcal{P}_a$ is Koszul, we have \[(M\mcal{P}_a^{\text{\textexclamdown}})_{\lambda, \mu} + 42M_{\lambda, \mu} + (M\mcal{P}_a)_{\lambda, \mu} = 9M_{\lambda, \mu} + y^a_{\lambda, \mu} + 20M_{\lambda, \mu},\] which is equivalent to \[199M_{\lambda, \mu} = x^a_{\lambda, \mu} + y^a_{\lambda, \mu} + k_{\lambda, \mu} + k^{\perp}_{\lambda', \mu'}+ r^a_{\lambda, \mu} + r^{a, \perp}_{\lambda', \mu'}.\]

First we can remind the values of $M_{\lambda, \mu}$, and display the values of $k_{\lambda, \mu}$ and $k_{\lambda', \mu'}^{\perp}$ given by \texttt{SageMath} computations, and $199M_{\lambda, \mu} - k_{\lambda, \mu} - k^{\perp}_{\lambda', \mu'}$ in Table \ref{SomeValues}.

\begin{table}[h!]
\caption{Values of $M_{\lambda, \mu}$, $k_{\lambda, \mu}$, $k^{\perp}_{\lambda', \mu'}$ and $199M_{\lambda, \mu} - k_{\lambda, \mu} - k^{\perp}_{\lambda', \mu'}$ depending on $\mu$}
\label{SomeValues}
\centering
\begin{tabular}{r|c|c|c|c|c}
$\mu$ & $(4)$ & $(3, 1)$ & $(2, 2)$ & $(2, 1, 1)$ & $(1, 1, 1, 1)$ \\
\hline
$M_{\lambda, \mu}$ & $1$ & $3$ & $2$ & $3$ & $1$ \\
\hline
$k_{\lambda, \mu}$ & $73$ & $219$ & $146$ & $219$ & $73$ \\
\hline
$k^{\perp}_{\lambda', \mu'}$ & $92$ & $276$ & $184$ & $276$ & $92$ \\
\hline
$199M_{\lambda, \mu} - k_{\lambda, \mu} - k^{\perp}_{\lambda', \mu'}$ & $34$ & $102$ & $68$ & $102$ & $34$
\end{tabular}
\end{table}

The only values remaining are $r_{\lambda, \mu}^a$ and $r_{\lambda', \mu'}^{\perp}$, which are ranks of the matrices $B_{\lambda, \mu}^a$ and $B_{\lambda, \mu}^{a, \perp}$.
		
In order to find an upper bound of the rank of a polynomial matrix $M(x_1, \dots, x_l)$ of size $n \times m$, one can compute, for $1 \leq i \leq \min(n, m)$, the determinantal ideal $DI_i(M)$ (see \cite[Section 3]{BrDo}), which is the ideal generated by all minors of $M$ of size $i$, and compute its Gröbner basis. The Gröbner basis is $0$ if and only if for all $(z_1, \dots, z_l) \in \bb{C}^l$, $\rk(M(z_1, \dots, z_l)) < i$. We can compute Gröbner bases for the polynomial matrix $B_{\lambda, \mu}^{a, \perp}$ in the variables $(a_1, a_2, a_3, a_4)$, and find the values of $r^{a, \perp}_{\lambda', \mu'}$ depending on $a$ and $\mu$, see Table \ref{rperpalambdamu}.
			
\begin{table}[h!]
\caption{$r^{a, \perp}_{\lambda', \mu'}$ depending on $a$ and $\mu$}
\label{rperpalambdamu}
\centering
\begin{tabular}{r|c|c|c|c|c}
$a\in \backslash \mu$ & $(4)$ & $(3, 1)$ & $(2, 2)$ & $(2, 1, 1)$ & $(1, 1, 1, 1)$ \\
\hline
$C$ & $0$ & $0$ & $0$ & $0$ & $0$ \\
\hline
$NC_{\norm}$ & $1$ & $3$ & $2$ & $3$ & $1$ \\
\hline
$NC_1$ & $1$ & $3$ & $2$ & $3$ & $0$ \\
\hline
$NC_{x, 0}$ & $1$ & $3$ & $2$ & $2$ & $0$
\end{tabular}
\end{table}
			
Thus, if $\mcal{P}_a$ is Koszul, we have \[r^a_{\lambda, \mu} = 199M_{\lambda, \mu} - k_{\lambda, \mu} - k^{\perp}_{\lambda', \mu'} - x^a_{\lambda, \mu} - y^a_{\lambda, \mu} - r^{a, \perp}_{\lambda', \mu'} =: \Sigma^a_{\lambda, \mu},\] and we have the values of $\Sigma^a_{\lambda, \mu}$ depending on $a$ and $\mu$ in Table \ref{sigmaalambdamu}.
			
\begin{table}[h!]
\caption{$\Sigma^a_{\lambda, \mu}$ depending on $a$ and $\mu$}
\label{sigmaalambdamu}
\centering
\begin{tabular}{r|c|c|c|c|c}
$a\in \backslash \mu$ & $(4)$ & $(3, 1)$ & $(2, 2)$ & $(2, 1, 1)$ & $(1, 1, 1, 1)$ \\
\hline
$C$ & $0$ & $0$ & $0$ & $0$ & $0$ \\
\hline
$NC_{\norm}$ & $6$ & $18$ & $12$ & $18$ & $6$ \\
\hline
$NC_1$ & $5$ & $16$ & $11$ & $13$ & $11$ \\
\hline
$NC_{x, 0}$ & $4$ & $9$ & $5$ & $7$ & $10$
\end{tabular}
\end{table}

For $r_{\lambda, \mu}^a$ however, this is more complicated because of the sizes of the matrices $B^a_{\lambda, \mu}$. Moreover, on couples of partitions giving small enough matrices, Gröbner bases show that if $a \in NC_{\norm}$, the Koszul criterion can be verified (for example $a = (2, 2, 0, 0)$). But there is still many cases we can look at.

\subsection{Conclusion}

By studying the possible values of $r_{\lambda, \mu}^a$ for some partitions $(\lambda, \mu) \vdash (2, 4)$, we can state several interesting theorems. For example for $(\lambda, \mu) = ((1, 1), (1, 1, 1, 1))$, we get \[r_{\lambda, \mu}^a = \begin{cases} 0 &\text{ if } a \in C \\ 3&\text{ if }a \in NC_1 \\ 2&\text{ if } a \in NC_{x, 0}\end{cases}.\] This proves the following.
			
\begin{theo} Let $a = (a_1, a_2, a_3, a_4) \in \{0, 1\}^4$. The following are equivalent :
			
\begin{enumerate}
\item the properad $\mcal{P}_a$ induces a confluent system,
\item the properad $\mcal{P}_a$ is Koszul.
\end{enumerate}
			
Moreover, in this case, $\between_a$ is one of the following up to isomorphism : \[\jesus{}{}{}{}, \jesus{}{}{}{} - \dromadroite{}{}{}{} \text{ or } \jesus{}{}{}{} - \dromadroite{}{}{}{} - \dromagauche{}{}{}{}.\]
\end{theo}
			
Nevertheless, the subfamily of properads considered in this theorem is finite of size $16$. Up to isomorphism, this family is of size $7$, thus this theorem is based on Koszulness of three different properads up to isomorphism, and proves non Koszulness of four different properads up to isomorphisms, the ones given by $\between_a$ equal to 
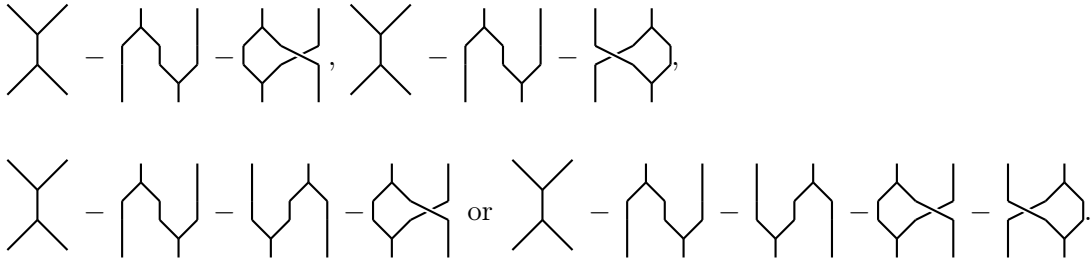
\begin{align*}
&\jesus{}{}{}{} - \dromadroite{}{}{}{} - \poissongauche{}{}{}{}, \jesus{}{}{}{} - \dromadroite{}{}{}{} - \poissondroite{}{}{}{}, \\
&\jesus{}{}{}{} - \dromadroite{}{}{}{} - \dromagauche{}{}{}{} - \poissongauche{}{}{}{}\text{ or } \jesus{}{}{}{} - \dromadroite{}{}{}{} - \dromagauche{}{}{}{} - \poissongauche{}{}{}{} - \poissondroite{}{}{}{}.
\end{align*}
			
We now use Gröbner bases via \texttt{SageMath} on the matrices $B_{\lambda, \mu}^a$ for $\lambda = (1, 1)$ and $\mu = (1, 1, 1, 1)$, but for subfamilies of $\bb{C}^4$ in order to partially prove the conjecture. For example, if we take $a_2 = a_4 = 0$, we have the following.
			
\begin{theo} \label{TheoFirst}
For $a = (a_1, 0, a_3, 0)$, the following are equivalent :
\begin{enumerate}
\item the properad $\mcal{P}_a$ is Koszul,
\item the relation $\between_a$ is one of the following, up to isomorphism : \[\jesus{}{}{}{}\text{ or }\jesus{}{}{}{} - \dromadroite{}{}{}{}.\]
\end{enumerate}
\end{theo}

We have similar results for many other subfamilies.

\begin{theo} \label{TheoGen} Conjecture \ref{Conjecture} is true for the subfamilies given by $(a_1, 0, a_3, 0)$, $(0, a_2, 0, a_4)$, $(a_1, 1, 0, 0)$, $(1, a_2, 0, 0)$, $(1, 1, a_3, 0)$ and $(1, 1, 0, a_4)$. \end{theo}
			
These results all prove non Koszulness of some properads of the form $\mcal{P}_a$ that induces a non confluent system, but not all of them. In order to prove Conjecture \ref{Conjecture}, one can look in different weights and biarities. For example, one can look in biarity $(3, 3)$ and weight $4$, the spaces considered are bigger, thus the calculations by hand and by computer take more time, but we may find another stronger criterion on $\mcal{P}_a$ to be Koszul. Another idea would be to study the other partitions $(\lambda, \mu) \vdash (2, 4)$ to find potential obstructions. This does not require many more calculations but the problem is that we have to compute minors of big size in matrices of bigger size. We could also go step by step, considering for example $a_1 \neq 0$ (because the problem holds in $NC_{\norm}$), thus we can divide by $a_1$ and simplify the matrix.
			
However these tools can be used to study families of parametrized properads in order to look at possibly Koszul properads, or to prove their non-Koszulness. In the case of non-quadratic properads, one can use these tools to prove (non) homotopy Koszulness of such properads by looking at their associated quadratic properads. However, the condition of $\bb{S}$-bimodule isomorphism between the properads and their associated quadratic properad seems not easy to refute.

\section{\texttt{Sagemath} Script}
\label{SecSagemath}

Here we present the idea behind the \texttt{SageMath} script used in Section \ref{SecKoszulness} and Section \ref{SecNonKoszulness}. The script can be found in \cite{NeWeb}, together with the files \texttt{isobiarity23Method2.ipynb} and \texttt{Koszulcpxbiarity24.ipynb}, the file \texttt{readme.md} that describes every module, and the file \texttt{tutorial.ipynb} that shows how to use the script on a simple example. In this appendix, we will focus on how the script works rather than describing the modules or showing how to actually use it. The script is meant to evolve, maybe even by changing the language.

\subsection{Construction of free properads and ideals}

To construct a free properad over generators, we give a list of graphs which are generators in weight $1$. In order to construct the graphs in weight $n$ from weight $n-1$, we take every element in weight $n-1$ and look at every possible way we can link a generator to this element. This constructs a list of graphs, but possibly with duplicates. That is why we need to remove all duplicate graphs from this list and this is what take the most time in this construction. The time taken by this step is why we save all generated properads, and later ideals, in files via the \texttt{pickle} module.

The idea for the free ideal is the same, we give a list of relations, a relation being a list of couples of the form (graph, coefficient), and a free properad in which this ideal lives. Then we generate, step by step, the space of relations with the same method as for free properads.

In order to study the family of properads in this paper, we generated a free properad over a product and a coproduct. In this properad, we generated an ideal with relations being associativity, coassociativity, and $\between_a$, $a$ being encoded as polynomial variables. Thus we can compute a generating family of the space of relations in some weight and biarity and study it. The issue here is that the generating family in weight $4$, the one that interests us, is way too big and we cannot determine the rank depending on $a$ that easily. Thus we need to use the divide and conquer method from M. Bremner and V. Dotsenko in \cite{BrDo}.

\subsection{The divide and conquer method}

Here the idea was to encode the divide and conquer method, we use exactly the same method as in \cite{BrDo}, but we just add the Kronecker product to the process because we work on properads instead of operads. The other difficulty was to get only one relation per orbit under the $\bb{S}_m \times \bb{S}_n^{\op}$ action, because the external outputs and inputs of similar graphs were in the same order. That is why we added to the function that removes duplicates a standardizing function, that uniformizes all graphs.

Once one relation per orbit chosen, we can compute the representation matrices of the space of relations in some weight and biarity and eventually compute the multiplicities of this representation in terms of $a$. In order to compute these multiplicities, one can use Gröbner bases or primary decompositions of the determinental ideals.

\subsection{Limits}

This script can be used to compute free properads or properads by generators and relations weight by weight (if the space of relations is homogeneous). For a specific properad, one can easily compute any desired dimension for a weight lower or equal to $4$. This may take a few hours for weight $4$. In the futur, we propose to study the properad $BiB^{\lambda}$ encoding balanced infinitesimal bialgebras, see \cite[Section 2]{Qu}, or the properad $V$ encoding $V$-gebras, see \cite[Section 3]{LeVa}.

However, there are some limits to this script. So far, handling symmetries on generators is not easy, for example for the properad $V$, one has to compute the free properad over a non-cocommutative bi-tensor, then compute by hand the consequences of cocommutativity of the bi-tensor and write down these additionnal relations, together with the other ones, in the space of relations. This increases the amount of calculations done by the script, thus slows it down a lot. Another limit of this script is that it is very long to compute weight $5$ of a free properad. Maybe it can be optimized.

\printbibliography
        
\end{document}